\documentclass[final,onefignum,onetabnum]{siamart220329}

\usepackage{lipsum}
\usepackage{amsfonts}
\usepackage{graphicx}
\usepackage{epstopdf}
\usepackage{algorithmic}
\ifpdf
  \DeclareGraphicsExtensions{.eps,.pdf,.png,.jpg}
\else
  \DeclareGraphicsExtensions{.eps}
\fi

\newsiamremark{remark}{Remark}
\newsiamremark{hypothesis}{Hypothesis}
\crefname{hypothesis}{Hypothesis}{Hypotheses}
\newsiamthm{claim}{Claim}

\headers{Regularized Convex Regression}{Chen and Mazumder}

\title{{Subgradient Regularized} Multivariate Convex Regression at Scale}

\author{Wenyu Chen\thanks{MIT Operations Research Center 
  (\email{wenyu@mit.edu}).}
\and Rahul Mazumder\thanks{MIT Sloan School of Management, Operations Research Center and MIT Center for Statistics
  (\email{rahulmaz@mit.edu}).}
}

\usepackage{amsopn}

\usepackage{amstext}
\usepackage{amsfonts}
\usepackage{amssymb}
\usepackage{xcolor,tcolorbox}
\usepackage{ulem,paralist}
\usepackage{mathtools}
\usepackage{bm}
\usepackage{mathrsfs}
\usepackage{amsmath}
\usepackage{thmtools,thm-restate}

\usepackage{verbatim}
\usepackage{booktabs}
\usepackage{footnote}

\newcommand{\B}{\boldsymbol}
\newcommand{\M}{\boldsymbol}

\renewcommand{\vec}{\mathrm{vec}}
  
\newcommand{\norm}[1]{\lVert#1\rVert}

\newcommand{\E}{\mathbb{E}}

\newcommand{\R}{\mathbb{R}}
\newcommand{\code}[1]{{\ttfamily #1}}

\renewcommand{\emph}{\it}

\begin{document}

\maketitle

\begin{abstract}
We present new large-scale algorithms for fitting a {subgradient regularized} multivariate convex regression function to $n$ samples in $d$ dimensions---a key problem in shape constrained nonparametric regression with applications in statistics, engineering and the applied sciences. 
The infinite-dimensional learning task can be expressed via a convex quadratic program (QP) with $O(nd)$ decision variables and $O(n^2)$ constraints. While instances with $n$ in the lower thousands can be addressed with current algorithms within reasonable runtimes, solving larger problems (e.g., $n\approx 10^4$ or $10^5$) is computationally challenging. 
To this end, we present an active set type algorithm on the dual QP. For computational scalability, we allow for approximate optimization of the reduced sub-problems; and propose randomized augmentation rules for expanding the active set.
{{We derive novel computational guarantees for our algorithms.}}
We demonstrate that our framework can {approximately} solve instances of the subgradient regularized convex regression problem with $n=10^5$ and $d=10$ within minutes; and shows strong computational performance compared to earlier approaches.  
\end{abstract}



\section{Introduction}\label{sec:intro}
{Given $n$ samples $(y_{i}, \M{x}_{i})$, $i \in [n]:=\{1, \ldots, n\}$ with response $y_{i} \in \mathbb{R}$ and 
covariates $\M{x}_{i} \in \mathbb{R}^d$, we consider the task of predicting $y$ using a convex 
function of $\M{x}$. This convex function is unknown and needs to be estimated from the data. This leads rise to the so-called multivariate {\emph{convex regression}} problem~\cite{seijo2011nonparametric,lim2012consistency} where we minimize the sum of squared residuals
\begin{equation}\label{convex-infinite-dim-est}
\hat{\varphi} = \arg\min_{\varphi\in \mathcal{F}}~\frac{1}{n}\sum_{i=1}^n(y_i-\varphi(\bm{x}_i))^2
\end{equation}
over all real-valued convex functions in $\mathbb{R}^d$, denoted by $\mathcal{F}$.
Above, $\varphi(\bm{x}_i)$ is the value of the convex function $\varphi$ at point $\M{x}_{i}$.

In the special case where $\varphi(\bm{x})=\M{x}^T\B\beta$ is a linear function, with unknown regression coefficients $\B\beta$, criterion~\eqref{convex-infinite-dim-est} leads to the well-known least squares problem. Problem~\eqref{convex-infinite-dim-est} is an instance of shape constrained nonparametric regression~\cite{robertson1988order}---here we learn the underlying function $\varphi$ under a qualitative shape constraint such as convexity. The topic of function estimation under shape constraints has received significant attention in recent years---see for example, the special issue in {\it Statistical Science}~\cite{samworth2018special} for a nice overview.}

Convex regression is widely used in economics, operations research, statistical learning and engineering applications. In economics applications, for example, convexity/concavity arise in modeling utility and production functions, consumer preferences~\cite{varian1982nonparametric,johnson2018shape}, among others.
In some stochastic optimization problems, value functions are taken to be convex~\cite{shapiro2009lectures}. See also the works of~\cite{hannah2013multivariate,wang2013estimating,balazs2016convex} for other important applications of convex regression.

There is a rich body of work in statistics studying different (statistical) methodological aspects of convex 
regression~\cite{seijo2011nonparametric,lim2012consistency,balazs2016convex,han2016multivariate,han2019convergence,kur2019optimality}. 
However, the challenges associated with computing the convex regression estimator limit our empirical understanding of this estimator and its usage in practice. More recently, there is a growing interest in developing efficient algorithms for this optimization problem---see for e.g.~\cite{mazumder2018computational,aybat2016parallelizable,bertsimas2021sparse,lin2020efficient}. The focus of the current paper is to further advance the computational frontiers of convex regression.

We note that the infinite-dimensional Problem~\eqref{convex-infinite-dim-est} can be reduced~\cite{seijo2011nonparametric,mazumder2018computational} to a finite dimensional convex quadratic program (QP):
\begin{equation}\label{eqn:nonreg}
\begin{aligned}
    &\underset{\phi_1,\ldots,\phi_n;\bm{\xi}_1,\ldots,\bm{\xi}_n}{\text{minimize}}&&\frac{1}{n}\sum_{i=1}^n(y_i-\phi_i)^2\\
    &\quad\,\text{s.t.}&& \phi_j-\phi_i\geq \langle \bm{x}_j-\bm{x}_i,\bm{\xi}_i\rangle,\quad \forall (i,j)\in\Omega
    \end{aligned}
\end{equation}
where $\phi_1,\ldots,\phi_n\in\R$, $\B\xi_1,\ldots,\B\xi_n\in\R^d$,  and $\Omega:=\{(i,j):1\leq i,j\leq n,i\neq j\}$. In~\eqref{eqn:nonreg}, $\phi_{i}=\varphi(\bm{x}_i)$ and $\bm\xi_i$ is a subgradient of $\varphi(\bm{x})$ at $\bm x = \bm x_i$. 
Problem~\eqref{eqn:nonreg} has $O(nd)$ variables and $O(n^2)$ constraints and becomes computationally challenging when $n$ is large.  
For the convex regression problem to be statistically meaningful~\cite{kur2019optimality,mazumder2018computational}, we consider cases with $n \gg d$ (and number of features $d\sim \log n$ to be small). Off-the-shelf interior point methods~\cite{seijo2011nonparametric} for~\eqref{eqn:nonreg} are limited to instances where $n$ is at most a few hundred. 
\cite{aybat2016parallelizable} consider a regularized version of~\eqref{eqn:nonreg} (i.e.,~\eqref{eqn:origin}, below) and propose parallel algorithms to solve instances with $n\approx 1,600$ leveraging commercial QP solvers (like Mosek). 
\cite{mazumder2018computational} use an alternating direction method of multipliers (ADMM)~\cite{boyd-admm1}-based algorithm
that can address problems up to $n \approx 3,000$.
Recently,~\cite{lin2020efficient} propose a different ADMM method and also a proximal augmented Lagrangian method  where the subproblems are solved by the semismooth Newton method---they address instances with $n\approx 3,000$. 
Algorithms based on nonconvex optimization have been proposed to learn convex functions that are representable as a piecewise maximum of $k$-many hyperplanes~\cite{hannah2013multivariate,balazs2016convex,ghosh2019max}---these are interesting approaches, but they may not lead to an optimal solution
for the convex regression convex optimization problem.

Recently,~\cite{bertsimas2021sparse} present a cutting plane or constraint generation-type algorithm for~\eqref{eqn:nonreg}: At every iteration, they solve a reduced QP by considering a subset of constraints. Leveraging capabilities of commercial solvers (e.g., Gurobi), this can approximately solve instances of~\eqref{eqn:nonreg} with $n\approx 10^4$-$10^5$. In this paper, we also present an active-set type method, but our approach differs from~\cite{bertsimas2021sparse}, as we discuss below. We also establish novel computational guarantees for our proposed approach.

The convex least squares estimator~\eqref{eqn:nonreg} may lead to undesirable statistical properties when $\bm{x}$ is close to the boundary of the convex hull of $\{\B{x}_{i}\}_{1}^{n}$~\cite{mazumder2018computational,balabdaoui2007consistent}. 
This problem can be improved 
by considering a subgradient regularized version of~\eqref{eqn:nonreg} given by:
\begin{align}
\begin{aligned}
    &\underset{\phi_1,\ldots,\phi_n;\bm{\xi}_1,\ldots,\bm{\xi}_n}{\text{minimize}}&&\frac{1}{n}\sum_{i=1}^n(y_i-\phi_i)^2+\frac{\rho}{n}\sum_{i=1}^n\norm{\bm{\xi}_i}^2 \\
    &\quad\,\text{s.t.}&& \phi_j-\phi_i\geq \langle \bm{x}_j-\bm{x}_i,\bm{\xi}_i\rangle,\quad \forall (i,j)\in\Omega
    \end{aligned}
    \tag{\mbox{$P_0$}}\label{eqn:origin}
\end{align}
where, we impose an additional $\ell_2$-based regularization on the subgradients $\{\bm\xi_{i}\}_{1}^n$ of $\varphi$; {{with
$\rho > 0$}} being the regularization parameter. Formulation~\eqref{eqn:origin} also appears in~\cite{aybat2016parallelizable}. Statistical estimation error properties of a form of subgradient regularized convex least squares estimator appear in~\cite{mazumder2018computational}. {{In the special case where $\varphi(\bm{x})=\M{x}^\top\B\beta$ is a linear function,~\eqref{eqn:origin} leads to the popular ridge regression estimator (i.e., least squares with an additional squared $\ell_2$ penalty on $\B\beta$). In ridge regression a nonzero penalty on $\|\B\beta\|_{2}^2$
often leads to improved statistical performance over vanilla least squares. Similarly, in convex regression, a value of $\rho>0$ can result in better statistical estimation error compared to the unregularized estimator with $\rho=0$---See Section~\ref{sec:experiments} for numerical support on a collection of datasets.}}

\medskip

\noindent {\bf Our approach:} 
In this paper, we focus on solving~\eqref{eqn:origin} for $\rho>0$. Problem~\eqref{eqn:origin} has a strongly convex objective function in the decision variables $(\{\phi_{i}\}_{1}^n,\{\bm\xi_{i}\}_{1}^n)$ --- its Lagrangian dual (see \eqref{eqn:main-dual} below) is a convex QP with $O(n^2)$ variables over the nonpositive orthant.
We present large scale algorithms for this dual and study their computational guarantees. 
The large number of variables poses computational challenges for full-gradient-based optimization methods as soon as $n$ becomes larger than a few thousand.
However, we anticipate that $O(n)$-many of the constraints in~\eqref{eqn:origin} will suffice. We draw inspiration from the one dimensional case ($d=1$), an observation that was also used by~\cite{bertsimas2021sparse}. Hence, we use methods inspired by constraint generation~\cite{bertsimas1997introduction}, which we also refer to (with an abuse of terminology) as active set type methods~\cite{bertsekas1997nonlinear}. Every step of our algorithm considers a reduced dual problem where the decision variables,  informally speaking,  correspond to a subset of the primal constraints.
The vanilla version of this active set method, which solves the reduced dual problem to {\emph{optimality}}\footnote{This is similar to the method of~\cite{bertsimas2021sparse} who consider a constraint generation method for~\eqref{eqn:nonreg} where the reduced sub-problems are solved to optimality.}, becomes expensive when $n$ and/or the size of the active set becomes large especially if one were to perform several active-set iterations. 
We propose improved algorithms  
that perform inexact (or approximate) optimization for this reduced problem initially and then increase the optimization accuracy at a later stage.
Upon solving the reduced problem (exactly or inexactly), we examine 
optimality conditions for the full problem; 
and include additional 
variables into the dual problem, if necessary. 

To augment the current active-set, greedy deterministic augmentation rules that scan all $O(n^2)$-constraints, become computationally expensive---therefore, we use randomized rules, which leads to important cost savings. These randomized augmentation rules extend the random-then-greedy selection strategies proposed by~\cite{lu2018randomized} in the context of Gradient Boosting Machines~\cite{friedman2001greedy}. 
Our approach operates on the dual and results in a dual feasible solution---we show how this leads to a primal feasible solution, delivering a duality gap certificate.

We establish a novel linear convergence rate of our algorithm {(in terms of outer iterations)} on the dual, which is not strongly convex. Our guarantees apply to both exact/inexact optimization of the reduced problem; and both deterministic and randomized augmentation rules. 
As we focus on large scale problems (e.g. $n\geq 10,000$),
inexact optimization of the reduced sub-problem and randomized augmentation rules play a key role in computational efficiency. As we carefully exploit problem-structure, our standalone algorithms enjoy a low memory footprint and can approximately solve instances of subgradient regularized convex regression  with $n \approx 10^5$ and $d=10$ in minutes.
Numerical comparisons suggest that on several datasets, our approach appears to notably outperform earlier algorithms in solving~\eqref{eqn:origin}
for values of $\rho>0$ that result in good statistical performance. For larger scale instances, our approach is better suited for finding solutions with low to medium accuracy. Since our approach is based on the smooth dual of~\eqref{eqn:origin}, the performance of our algorithm would deteriorate when $\rho$ is numerically very close to zero. In particular, our approach may not be suitable for obtaining a high-accuracy solution to the unregularized problem~\eqref{eqn:nonreg}.

\medskip

\noindent {\bf Organization of paper:} \Cref{sec:problem} presents both primal and dual formulations of the full problem~\eqref{eqn:origin}; and a first order method on the dual.
\Cref{sec:algo} presents active-set type algorithms, augmentation rules and associated computational guarantees. \Cref{sec:dualgap} discusses computing duality gap certificates. 
\Cref{sec:experiments} presents numerical experiments.
Some technical details are relegated to \Cref{app:proofs} to improve readability.

\medskip

\noindent {\bf Notations:} For convenience, we list some notations used throughout the paper.
We denote the set $\{1,2,\ldots, n\}$ by 
$[n]$. The cardinality of a set $W$ is denoted by $|W|$. We denote by $\mathbb{R}_+$, $\mathbb{R}_{++}$ the set of nonnegative and positive real numbers (respectively). A similar notation applies for $\mathbb{R}_{-}$ and $\mathbb{R}_{--}$.
Symbols $\bm{1}_n$, $\bm{e}_i$ and $\bm{I}$ denote: a vector of length $n$ of all ones, the $i$-th standard basis element and the identity matrix (respectively). 
$\mathrm{span}(A)$ denotes the linear space generated by the vectors in the set $A$. For a matrix $\bm{B}$, let $\vec(\bm{B})$ denote a vectorized version of $\bm{B}$.
The largest singular value of a matrix $\M{B}$ is denoted by $\lambda_{\max}(\bm{B})$.
We use $\norm{\cdot}$ to denote the Euclidean norm of a vector and the spectral norm of a matrix. 
Finally, $\partial f(\bm{x})$ denotes {the subdifferential (set of subgradients)} of $f$ at $\bm{x}$. 

\section{Primal and Dual Formulations} \label{sec:problem}
We introduce some notation to rewrite Problem~\eqref{eqn:origin} compactly. Let $\bm{y}=[y_1,\ldots,y_n]^\top\in\mathbb{R}^n$, $\bm{X}=[\bm{x}_1^\top,\ldots,\bm{x}_n^\top]$ $\in\mathbb{R}^{n\times d}$, $\bm{\phi}=[\phi_1,\ldots,\phi_n]^\top\in\mathbb{R}^n$, and $\bm{\xi}=\vec([\bm{\xi}_1,\ldots,\bm{\xi}_n])\in\mathbb{R}^{nd}$. We define $\bm A\in\R^{n(n-1)\times n}$ and $\bm B\in\R^{n(n-1)\times nd}$ such that the rows of $\bm A\bm\phi+\bm B\bm\xi$ correspond to the constraints $\phi_j-\phi_i-\langle\bm x_j-\bm x_i,\bm\xi_i\rangle$ for $(i,j) \in \Omega$. 
Hence \eqref{eqn:origin} is equivalent to:
\begin{align}\label{eqn:main-primal}
\begin{aligned}
    &\underset{\bm{\phi},\bm{\xi}}{\text{minimize}}&&f(\bm{\phi},\bm{\xi}):=\frac{1}{2}\norm{\bm{y}-\bm{\phi}}^2+\frac{\rho}{2}\norm{\bm{\xi}}^2~~~~~\text{s.t.}& &\bm{A\phi}+\bm{B\xi} \geq \M{0},
    \end{aligned}
    \tag{$P$}
\end{align}
where, $\bm{A\phi}+\bm{B\xi} \geq \M{0}$ denotes componentwise inequality.
Due to strong convexity of the objective,~\eqref{eqn:main-primal} has a unique minimizer denoted by $(\bm{\phi}^\star,\bm{\xi}^\star)$. While~\eqref{eqn:main-primal} has a large number (i.e., $O(n^2)$) of constraints, given that the affine hull of the data points has full dimension\footnote{This occurs with probability one if the covariates are drawn from a continuous distribution.}, i.e. $\mathrm{span}(\{\bm x_i-\bm x_j\}_{j\neq i})=\R^d$, it can be shown that the constraint matrix $\bm{C}=\begin{bmatrix}\bm{A}&\bm{B}\end{bmatrix}\in\mathbb{R}^{n(n-1)\times(n+nd)}$ is of rank $O(nd)$. 
This serves as a motivation for our active-set approach.

The Lagrangian dual of~\eqref{eqn:main-primal} is equivalent to the following convex problem
\begin{equation}L^\star  = \underset{\bm{\lambda}\in\mathbb{R}^{n(n-1)}}{\mathrm{minimize}}\quad L(\bm{\lambda}):=\frac{1}{2\rho}\bm{\lambda}^\top(\rho\bm{AA}^\top+\bm{BB}^\top)\bm{\lambda}- \bm{y}^\top\bm{A}^\top\bm{\lambda} \quad \mathrm{s.t.}\quad\bm{\lambda}\leq 0.\tag{$D$}\label{eqn:main-dual}\end{equation}
{\begin{definition}
A convex function $f$ is $\sigma$-smooth if it is continuously differentiable with $\sigma$-Lipschitz gradient; $f$  is $\mu$-strongly convex if $f(\bm{x})-\frac{\mu}{2}\bm{x}^\top\bm{x}$ is convex. 
\end{definition}}
Note that $\B\lambda \mapsto L(\bm{\lambda})$ is not strongly convex, but it is $\sigma$-smooth, where $\sigma$ is the maximum eigenvalue of the matrix $\bm Q:=\bm A\bm A^\top+\frac{1}{\rho}\bm B\bm B^\top$.

Unlike the primal~\eqref{eqn:main-primal}, 
the dual problem~\eqref{eqn:main-dual} is amenable to proximal 
gradient methods~\cite{nesterov2018lectures,beck2009fast} (PGD).
Other gradient based methods like accelerated proximal gradient methods (APG) \cite{beck2009fast,nesterov2018lectures}, the limited-memory Broyden-Fletcher-Goldfarb-Shanno (L-BFGS) method~\cite{liu1989limited} (for example), may also be used; and they work well in our numerical experience.  

However, every iteration of PGD requires computing the gradient $\nabla L(\M{\lambda}) \in \mathbb{R}^{n(n-1)}$. While an unstructured gradient computation will cost $O(n^4)$, exploiting the structure of $\M{A}, \M{B}$, this cost can be reduced to $O(n^2d)$, 
allowing us to scale these algorithms for instances with $n \approx 3,000$. 
Such matrix-vector multiplications can be used to estimate $\sigma$ via the power method or backtracking line-search~\cite{beck2009fast}. 
When $L$ is $\sigma$-smooth, PGD enjoys a standard sublinear convergence rate $O(1/t)$~\cite{beck2009fast}. With an additional strong convexity 
assumption PGD is known to converge at a linear rate \cite{nesterov2013gradient}. Note $L(\bm\lambda)$ is not strongly convex.
However, $L(\bm\lambda)$ satisfies the Polyak-{\L}ojasiewicz condition \cite{karimi2016linear,necoara2018linear} under which PGD converges at a linear rate. We note more general convergence results under error bound conditions can be found in~\cite{luo1992linear,luo1993error,li1995error, bolte2017error}.

\section{Active Set Type Algorithms}\label{sec:algo}
As~\eqref{eqn:main-dual} has $O(n^2)$ variables, the proximal gradient method (owing to full gradient computations) becomes prohibitively expensive when $n$ becomes larger than a few thousand. 
However, as discussed earlier, we {expect only} $O(n)$-many of these variables to be nonzero at an optimal solution---motivating the use of
a {constraint-generation/active set-type} method on the primal~\eqref{eqn:main-primal}, which relates to a column-generation type method on the dual~\eqref{eqn:main-dual}.

Constraint generation is traditionally used in the context of solving large-scale linear programs~\cite{dantzig1960decomposition,bertsimas1997introduction}. When used in the context of the QP~\eqref{eqn:main-primal}, as done in~\cite{bertsimas2021sparse}, we start with a reduced problem 
with a small subset of constraints in~\eqref{eqn:main-primal}. With a slight abuse of nomenclature, we refer to these constraints as the active set\footnote{Our usage of ``active set'' differs from the active set method for solving a QP, as discussed in~\cite{nocedal2006numerical}}.
After obtaining an optimal dual solution to this reduced problem, the traditional form of constraint generation will augment the active set with some dual variables that correspond to the violated primal constraints (if any) and re-solve the problem on the expanded set of constraints. 
We mention two shortcomings of this approach:
(a) Solving the reduced problem to optimality becomes expensive (especially when the active set becomes large and/or if several iterations of constraint-generation is needed); and (b)~finding variables to be appended to the active set has a large cost of $O(n^2d)$ operations.

To circumvent these shortcomings, we propose modifications to the above constraint generation or active set method. To address (a), we solve the reduced sub-problem inexactly (e.g., by taking a few iterations of the proximal gradient method). To address (b), we consider randomized rules to reduce the cost of augmenting the active set from $O(n^2d)$ to $O(nd)$ (for example).
We show that our proposed algorithm converges; and does so with a linear convergence rate in the outer iterations.

\subsection{Properties of the reduced problem}\label{subsec3:properties}

Let $W \subseteq \Omega$ index a subset of the constraints in~\eqref{eqn:main-primal}; and consider the reduced primal:
\begin{align}
\begin{aligned}
&\underset{\bm{\phi},\bm{\xi}}{\text{minimize}}&&f(\bm{\phi},\bm{\xi})=\frac{1}{2}\lVert\bm{y}-\bm{\phi}\rVert^2+\frac{\rho}{2}\lVert\bm{\xi}\rVert^2~~~~~\text{s.t.}&&\bm{A}_W\bm{\phi}+\bm{B}_W\bm{\xi}\geq \M{0}
\end{aligned}
\tag{\mbox{$P_W$}}
\label{eqn:relaxed-primal}
\end{align}
where,  $\bm{A}_W$ (and $\B{B}_{W}$) denotes matrix $\M{A}$ (and $\M{B}$) restricted to rows indexed by $W$.

{In the rest of the paper, we will use $W$ as a subscript for vectors or matrices whose size changes with $W$, and use $W$ (or $[W]$) as superscript for vectors or matrices whose size does not change with $W$. When $W=\Omega$, the relaxed problem is the original problem, and we drop the use of $\Omega$ as 
subscript and/or superscript for notational convenience.} \\
We consider solving the dual of~\eqref{eqn:relaxed-primal}.
\Cref{prop:relaxed-dual-char} presents some of its properties.
\begin{proposition}\label{prop:relaxed-dual-char}
The Lagrangian dual of \eqref{eqn:relaxed-primal} is given by:
\begin{equation}\label{eqn:relaxed-dual}
   \underset{\bm{\lambda}_W}{\mathrm{min}}~~ L_W(\bm{\lambda}_W):=\tfrac{1}{2\rho}\bm{\lambda}_W^\top(\rho\bm{A}_W\bm{A}_W^\top+\bm{B}_W\bm{B}_W^\top)\bm{\lambda}_W- \bm{y}^\top\bm{A}_W^\top\bm{\lambda}_W ~~\mathrm{s.t.}~~\bm{\lambda}_W\leq \M{0}\tag{\mbox{$D_W$}}.
\end{equation}
Let  $L_W^\star$ be the optimal objective value for~\eqref{eqn:relaxed-dual}.
The objective function $L_W(\cdot):\mathbb{R}^{|W|}\to\mathbb{R}$ is $\sigma_W$-smooth for some $\sigma_W\leq \sigma$. The set of all optimal solutions to \eqref{eqn:relaxed-dual} is a polyhedron of the form
\begin{equation}\Lambda^{\star}_W=\left\{\bm{\lambda}_W\in\mathbb{R}^{|W|}:\bm{A}_W^\top\bm{\lambda}_W=\bm{s}_A^W,\frac{1}{\sqrt{\rho}}\bm{B}_W^\top\bm{\lambda}_W=\bm{s}_B^W,\bm{\lambda}_W\leq 0\right\}\label{eqn:opt-dual-poly-W}\end{equation}
where, $\bm{s}_A^W = \bm{y}-\bm{\phi}^W_\star,\bm{s}_B=-\sqrt{\rho}\bm{\xi}^W_\star$; and $(\bm{\phi}_\star^W,\bm{\xi}_\star^W)$ is the unique optimal solution to the reduced
primal problem~\eqref{eqn:relaxed-primal}.
\end{proposition}
\begin{proof}[Proof Sketch]
{The Lagrangian dual is obtained by computing $\min_{\bm\xi,\bm\phi}\mathcal{L}(\bm\phi,\bm\xi;\bm\lambda_W)$, where $\mathcal{L}(\bm\phi,\bm\xi;\bm\lambda_W)=f(\bm\phi,\bm\xi)+\bm\lambda_W^\top(\bm A_W\bm\phi+\bm B_W\bm\xi)$ for any $\bm\lambda_W\leq 0$. The optimality (KKT) conditions for $\bm\phi,\bm\xi$ are $\bm\phi_\star^W=\bm y+\bm{A}_W^\top\bm\lambda_W$ and $\bm\xi_\star^W=\rho\bm B_W^\top\bm\lambda_W$. Plugging these equations into $\mathcal{L}$ and flipping the sign of the function lead to the objective $L_W$. As $\B\lambda_{W} \mapsto L_W(\B\lambda_{W})$ is a quadratic function, it is $\sigma_W$-smooth with $\sigma_W$ being the maximum eigenvalue of the Hessian matrix. Since the Hessian of the reduced dual problem is a submatrix of that of the original dual problem (corresponding to $W$ out of $\Omega$), it follows that $\sigma_W\leq \sigma$. The formulation of the polyhedral set follows from the KKT conditions~\cite{BV2004}. }
\end{proof}

PGD is readily applicable to~\eqref{eqn:relaxed-dual}. The per iteration cost is $O(|W|d)$, which is much smaller than 
$O(n^2d)$ (as $|W| \sim n$). Solving the reduced problem~\eqref{eqn:relaxed-dual} is usually much 
faster than solving the full problem~\eqref{eqn:main-dual} when $|W| \ll |\Omega|$.

\subsection{Augmentation Rules}\label{subsec:active-set-aug}

For any $W\subseteq \Omega$, given a feasible solution\footnote{This can be obtained by solving~\eqref{eqn:relaxed-dual} exactly (i.e., to optimality) or inexactly (i.e., approximately).} $\bm\lambda_W\in\R_{-}^{|W|}$ to the reduced dual~\eqref{eqn:relaxed-dual}, we can construct its corresponding primal variables $(\bm{\phi}^W,\bm{\xi}^W)$ for \eqref{eqn:relaxed-primal}, by making use of the KKT conditions:
\begin{equation}\label{def:primal-dual}
(\bm{\phi}^W,\bm{\xi}^W)
=(\bm{y}-\bm{A}^\top_W\bm{\lambda}_W,-\frac{1}{\rho}\bm{B}_W^\top\bm{\lambda}_W)\in\mathbb{R}^n\times\mathbb{R}^{nd}.
    \end{equation}
    
Notice that for a general dual variable $\bm\lambda_W\in\mathbb{R}^{|W|}_{-}$, the primal variables obtained via \eqref{def:primal-dual} may not be feasible for the reduced primal problem \eqref{eqn:relaxed-primal}. Below we discuss some rules for augmenting the current set of constraints (i.e., the active set). 

After obtaining $(\bm{\phi}^W,\bm{\xi}^W)$, we check if it is feasible for \eqref{eqn:main-primal}---that is, we verify if each component of $\bm{A\phi}^W+\bm{B\xi}^W$ (denoted by $v_{(i,j)}$) is nonnegative:
\begin{equation}\label{eqn:viols}
v_{(i,j)}=\phi_j^W-\phi_i^W-\langle \bm x_j-\bm x_i,\bm\xi_i^W\rangle \geq 0~~~\forall (i,j) \in \Omega.
\end{equation}
To this end, it is helpful to decompose the $O(n^2)$ constraints into $n$ blocks
$$\Omega = \bigcup_i\Omega_{i\cdot}=\bigcup_j\Omega_{\cdot j},$$
where $\Omega_{i\cdot} = \{(i,j):j\neq i,1\leq j\leq n\}$ and $\Omega_{\cdot j}=\{(i,j):i\neq j,1\leq i \leq n\}.$ Similarly, we define the slices $W_{i \cdot}$ and $W_{\cdot i}$ for all $i$.
The two decompositions $\{\Omega_{i\cdot}\}_{1}^{n}$ and $\{\Omega_{\cdot j}\}_{1}^{n}$ have different geometric interpretations. 
Now define points $P_j=\{\bm x_j,\phi_j^W\}$ and hyperplanes $H_i:y=\langle \bm{x}-\bm{x}_i,\bm\xi^W_i\rangle+\phi_i^W$ in $\mathbb{R}^{d+1}$. Note that each point $P_i$ lies on the hyperplane $H_i$, and $v_{(i,j)}$ denotes the vertical distance between $P_j$ and $H_i$. 
For each $i$, the nonnegativity of $v_{(i,j)}$ for all $(i,j) \in \Omega_{i\cdot}$, checks if the hyperplane $H_i$ lies below each point $P_j$, i.e. if $H_i$ is a supporting hyperplane of the points $\{P_j\}_1^n$. 
On the other hand, for each $j$, the nonnegativity of $v_{(i,j)}$ for all $(i,j) \in \Omega_{\cdot j}$, checks if $P_j$ lies above each hyperplane $H_i$---i.e., if $P_j$ lies above the piecewise maximum of these hyperplanes. The quantity $|\min_i v_{(i,j)}|$ is the amount by which $P_j$ lies below  the piecewise maximum. 

In \Cref{sec:dualgap} we see that the block decomposition interpretation is useful in obtaining a primal feasible solution.
Next we discuss a deterministic augmentation rule---due to its high computational cost, we subsequently present randomized augmentation rules---both of which make use of the above decomposition of $\Omega$. 

\subsubsection{Deterministic Augmentation Rule}\label{sec:det-augment}
We first present a simple greedy-like deterministic augmentation rule:

\smallskip

\begin{compactitem}
\item[{\bf Rule~{1}}.] \textbf{Greedy within each Block:}  For each block $\Omega_{i\cdot}$ (or $\Omega_{\cdot j}$),  choose $P$ pairs with the smallest $v_{(i,j)}$-values among $\Omega_{i\cdot}\backslash W_{i\cdot}$ (or $\Omega_{\cdot j}\backslash W_{\cdot j}$). From these $P$ indices, we add to the current active set $W$, only those with negative $v_{(i,j)}$-values.
\end{compactitem}

\smallskip

Rule~1 evaluates $O(n^2)$ constraints with computational cost {$O(n^2d+d|W|)$} and augments $W$ by at most $nP$ such constraints. For every block, obtaining the largest $P$ violations can be done via a partial sort---leading to a total cost of $O(n^2d + nP\log P)$ for $n$ blocks\footnote{A similar augmentation rule with $P=1$ is used in~\cite{bertsimas2021sparse}}. As a result, with this greedy rule, the augmentation becomes a major computational bottleneck for the active-set type method. This motivates the randomized augmentation rules discussed below.

\subsubsection{Randomized Augmentation Rules}\label{sec:rand-augment}
We present four randomized rules that sample a small subset of the  
indices in $\Omega \setminus W$ instead of performing a computationally intensive full scan across $O(n^2)$ indices as in Rule~1.

\smallskip

\begin{compactitem}
    \item[\bf Rule 2.] \label{crt2:random}\textbf{Random:} Uniformly sample $K$ indices in $\Omega\backslash W$---the cost of computing the corresponding $v_{(i,j)}$-values is $O(Kd+d|W|)$.
        
    \smallskip

    \item[\bf Rule 3.] \label{crt3:random-block}\textbf{Random within each Block:} For each block $\Omega_{i\cdot}$ (or $\Omega_{\cdot j}$), uniformly sample $P$ elements in $\Omega_{i\cdot}\backslash W_{i\cdot}$ (or $\Omega_{\cdot j}\backslash W_{\cdot j}$); the total computational cost is $O(nPd+d|W|)$.

    \smallskip

    \item[\bf Rule 4.] \label{crt4:RtG}\textbf{Random then Greedy:} Uniformly sample $M$-many $(i,j)$-pairs from $\Omega\backslash W$, and from these pairs choose the $K$ pairs with the smallest 
    $v_{(i,j)}$-values. Computing the $M$ values of $v_{(i,j)}$ cost $O(Md+d|W|)$, and the greedy selection step costs $M+K\log K$. The total computational cost is $O(Md+d|W|)$.
    
    \smallskip

    \item[\bf Rule 5.] \label{crt5:block-RtG}\textbf{Random Blocks then Greedy within each Block:} Uniformly sample $G$ groups from $\{\Omega_{i\cdot}\}_{i=1}^n$ (or $\{\Omega_{\cdot j}\}_{j=1}^n$) and for each group, choose the $P$ pairs that have the smallest $v_{(i,j)}$ values with $(i,j) \in \Omega_{i\cdot}\backslash W_{i\cdot}$ (or $\Omega_{\cdot j}\backslash W_{\cdot j}$). Similar to Rule 4, the total computational cost is $O(Gnd+d|W|)$.
\end{compactitem}

\medskip

Note that the above rules lead to a set of indices denoted by 
$\Delta'$. From these candidates, we only select those with negative $v_{(i,j)}$-values, which are then appended to the current active set $W$. That is,  if $\Delta=\{(i,j)\in\Delta':v_{(i,j)}<0\}$ then we set $W \leftarrow W \cup \Delta$.

Of the above, Rules~4 and 5 are inspired by, but different from the random-then-greedy coordinate descent procedure~\cite{lu2018randomized}, proposed in the context of Gradient Boosting.

\subsubsection{Norms Associated with the Augmentation Rules}
The computational guarantees of our active set-type algorithms depend upon certain norms induced by 
the aforementioned augmentation rules.
We first present some notation that we will use in our analysis.

\begin{definition}\label{def:indexset}
Given a vector $\bm{\theta}$, an index set $S$, and $k\leq |S|$, let $\mathcal{G}[S,k,\bm \theta]$ be the set of $k$ elements with the largest values of $|\theta_\omega|$ for $\omega\in S$. Given a pair  $(S,k)$ as above, we let $\mathcal{U}[S,k]$ denote the set of $k$ uniformly subsampled indices from the set $S$.
\end{definition}

\begin{definition}\label{def:rule-indexset}
Given a vector $\bm{\theta}\in\R^{n(n-1)}$ indexed by $\Omega$, let $\{S_i\}_{1}^{n}$ be disjoint subsets of $\Omega$, with $\bar{S}=\cup_{i \in [n]} S_{i}$.
Given positive integers $P,K,M,G$, we define $\delta_{1}(\bm\theta,\{S_i\}), \ldots, \delta_{5}(\bm\theta,\{S_i\}) \subset \Omega$ as follows:
\begin{eqnarray*}
\delta_{1}(\bm\theta,\{S_i\})=\bigcup\limits_{i=1}^n\mathcal{G}[S_i,P,\bm\theta],~~~
\delta_2(\bm\theta,\{S_i\})=\mathcal{U}[\bar{S},K],~~~~
\delta_3(\bm\theta,\{S_i\}) = \bigcup_{i=1}^n\mathcal{U}[S_i,P]\\
\delta_4(\bm\theta,\{S_i\})=\mathcal{G}[\mathcal{U}[\bar{S},M],K,\bm\theta],~~~~~~~
\delta_5(\bm\theta,\{S_i\})=\bigcup_{i\in\mathcal{U}[[n],G]}\mathcal{G}[S_i,P,\bm\theta].
\end{eqnarray*}
\end{definition}

With the above notation, we can express the indices to be augmented as a function of the violations ${\bm v}$, a vectorized representation of the entries $\{v_{(i,j)}\}_{(i,j)\in \Omega}$~\eqref{eqn:viols}. Let $\Delta'_{\{\ell\}}$ denote the pairs selected by Rule~$\ell$, and 
$\Delta_{\{\ell\}}=\{\omega\in\Delta_{\{\ell\}}':v_{(i,j)}<0\}\subseteq \Delta_{\{\ell\}}'$ be the final set to be added to $W$.
Let $\tilde{\bm v}$ be a vector having the same length as $\M{v}$, with its entries given by
$\tilde{v}_{(i,j)}=\min\{v_{(i,j)},0\}$. Then it is easy to verify that $\Delta'_{\{\ell\}}$ can be written as $\delta_\ell(\tilde{\bm v},\{S_i\})$ with $S_i=\Omega_{i\cdot}\setminus W_{i\cdot}$ or $S_i = \Omega_{\cdot i}\setminus W_{\cdot i}$.

\begin{definition}\label{def:rule-norms}
Given $\bm\theta\in\R^{n(n-1)}$, let $\{\Omega_i\}_{1}^n$ be either $\{\Omega_{i\cdot}\}_{1}^{n}$ or $\{\Omega_{\cdot i}\}_{1}^{n}$.
Define 
\begin{equation}\label{norm-theta-ell}
\norm{\bm\theta}_{\{\ell\}}=\left(\E\left[\sum_{\omega\in\delta_\ell(\bm\theta,\{\Omega_i\})}|\theta_\omega|^2\right]\right)^{1/2} ~~~\text{for}~~~\ell \in \{ 1, \ldots, 5\}.
\end{equation}
For $\ell=1$, there is no randomness in $\delta_\ell$, so the expectation can be removed. For $\ell\geq 2$, the expectation is taken over the randomness arising from the selection operator $\mathcal{U}$ (cf \Cref{def:indexset}).
\end{definition}

\Cref{lem:outer-conv-lem2-main} shows that the expression in display~\eqref{norm-theta-ell} leads to a norm on $\R^{n(n-1)}$. Furthermore, the norm equivalence constants appearing in~\eqref{eqn:alphal-betal} determine the convergence rates of \hyperlink{algo}{Algorithm~1} (see \Cref{thm:outer-conv}). 

\begin{lemma}\label{lem:outer-conv-lem2-main}
$\norm{\bm\theta}_{\{\ell\}}$ defined in \eqref{norm-theta-ell} is a norm on $\R^{n(n-1)}$. Furthermore, {the constants $\alpha_{\{\ell\}}, \beta_{\{\ell\}}$ listed in \Cref{table:summary}} satisfy that for all $\bm{x} \in \mathbb{R}^{n(n-1)}$
\begin{equation}\label{eqn:alphal-betal}
\alpha_{\{\ell\}}\norm{\bm{x}}_{\{\ell\}}^2\geq\norm{\bm{x}}_2^2\geq \beta_{\{\ell\}}\norm{\bm{x}}_{\{\ell\}}^2.
\end{equation}
\end{lemma}
The proof of \Cref{lem:outer-conv-lem2-main} appears in Section~\ref{subsecA:proof-lem2}.

\Cref{table:summary} presents a summary of some key characteristics of the five rules.

\begin{table}[h]\centering
\begin{tabular}{c|ccccc}
     \hline
     Rule& $\Delta'_{\{\ell\}}$&$|\Delta'_{\{\ell\}}|$&Augmentation Cost&$ \alpha_{\{\ell\}}$ &  $\beta_{\{\ell\}}$\\
     \hline
     1& $\cup_{i=1}^n\mathcal{G}[\Omega_i,P,\tilde{\bm v}]$&$nP$&$O(n^2d+d|W|)$&$\frac{n-1}{P}$&1\\
     2&$\mathcal{U}[\Omega,K]$&$K$&$O(Kd+d|W|)$&$\frac{n(n-1)}{K}$&$\frac{n(n-1)}{K}$\\
     3& $\cup_{i=1}^n\mathcal{U}[\Omega_i,P]$&$nP$&$O(nPd+d|W|)$&$\frac{n-1}{P}$&$\frac{n-1}{P}$\\
     4&$\mathcal{G}[\mathcal{U}[\Omega,M],K,\tilde{\bm v}]$&$K$&$O(Md+d|W|)$&$\frac{n(n-1)}{K}$&$\frac{n(n-1)}{M}$\\
     5&$\cup_{i\in\mathcal{U}[[n],G]}\mathcal{G}[\Omega_i,P,\tilde{\bm v}]$&$GP$&$O(Gnd+d|W|)$&$\frac{n(n-1)}{GP}$&$\frac{n}{G}$\\
     \hline
\end{tabular}
\caption{{\small Summary of some properties of the augmentation rules. Recall that $\Delta'_{\{\ell\}}$ denotes the pairs selected as per Rule~$\ell$. 
The number of candidates to be augmented to the current active set depends upon the signs of $v_{(i,j)}$s; and is of size at most $|\Delta'_{\{\ell\}}|$. 
Here, \textit{Augmentation Cost} is the cost of obtaining $\Delta'_{\{\ell\}}$. We present estimates of the norm-equivalence constants $\alpha_{\{\ell\}},\beta_{\{\ell\}}$~\eqref{eqn:alphal-betal}.
For notational convenience, we assume that $W=\varnothing$---for a nonempty $W$, we can replace $\Omega_i$ with $\Omega_i\backslash W_i$.} }

\label{table:summary}
\end{table}

\subsection{Active set method with inexact optimization of sub-problems}\label{sec:CG-inexact-method}
Once we augment the active set (based on Rules~1--5), we solve the updated reduced QP either exactly or inexactly. We then identify additional constraints to be added to the current active set; and continue till some convergence criteria is satisfied. Our algorithm is summarised below:

\begin{savenotes}
\begin{tcolorbox}[colback=white]
\begin{compactitem}
\item[] \underline{{\bf{\hypertarget{algo}{Algorithm~1}:}} An Active Set Type Method on the Dual~\eqref{eqn:main-dual}} 
\smallskip
\item[]  Initialize with $W^0$ and $\bm{\lambda}^0$. Perform the following steps~1--3 till convergence.
\item[1.] 
Obtain a feasible solution $\bm{\lambda}_{W^m}$ for $(D_{W^m})$
by either (a) solving $(D_{W^m})$ to optimality (i.e., exactly) or (b) solving $(D_{W^m})$ inexactly via a few steps of proximal gradient descent\footnote{Other choices are also possible, as discussed in Remark~\ref{remark-inexact-compute}.}.

\item[2.] Compute $\bm{\phi}^{W^m}=\bm{y}-\bm{A}^\top_{W^m}\bm{\lambda}_{W^m},$~ $\bm{\xi}^{W^m}=-\frac{1}{\rho}\bm{B}_{W^m}^\top\bm{\lambda}_{W^m}$ as per~\eqref{def:primal-dual}.
\item[3.] Use one of the Rules~1--5 to augment $W^{m}$ to obtain $W^{m+1}$; and go to Step~1.
\end{compactitem}
\end{tcolorbox}
\end{savenotes}

In what follows, we call \hyperlink{algo}{Algorithm~1} with option (a) [Step~1] the \textit{Exact Active Set} method (EAS); and option (b) [Step~1] the \textit{Active Set Gradient Descent} (ASGD) method.

If the size of the active set remains sufficiently small, interior point solvers (including heavy-duty commercial solvers like Gurobi, Mosek) may be used for full minimization of the reduced problem, as long as the number of active set iterations remain small. However, for larger problems (e.g., when 
$n\approx 10^4$--$10^5$), first order methods may be preferable due to warm-start capabilities (across active sets) and low memory requirements (by exploiting problem-structure). Furthermore, they also allow for inexact computation for the sub-problems---see numerical experiments in Section~\ref{sec:experiments} showing the important benefits in approximate optimization.

\Cref{thm:outer-conv} establishes that \hyperlink{algo}{Algorithm~1} requires at most 
$m=O(\log(1/\epsilon))$-many outer iterations to deliver an $\epsilon$-optimal dual solution for Problem~\eqref{eqn:main-dual}:
\begin{theorem}\label{thm:outer-conv}
For a given $\bm{\lambda}^0$ and $W^0$, let $\bm{\lambda}^m = \bm{\lambda}^{[W_m]}$ (for $m \geq 1$) be the sequence generated by \hyperlink{algo}{Algorithm~1}. Then, for either EAS or ASGD with augmentation Rule~$\ell$, we have:
$$\E[L(\bm{\lambda}^m)-L^\star]\leq\left(1-\frac{\mu}{\sigma\alpha_{\{\ell\}}}\right)^m [L(\bm{\lambda}^0)-L^\star],$$
where $\sigma$ is the smoothness constant of $L$, $\mu$ is a constant depending\footnote{{See Section~\ref{app:lemma-5-helper-thm} for a characterization of $\mu$}.} only on $\bm{A},\bm{B}$ and $\rho$, and $\alpha_{\{\ell\}}$ appears in \Cref{lem:outer-conv-lem2-main}.
\end{theorem}

The proof of \Cref{thm:outer-conv} is provided in Sections~\ref{app:lemma-5-helper-thm} and~\ref{subsec:proof-outer-thm}. We make a few remarks regarding \Cref{thm:outer-conv}.

\begin{remark}
If we used PGD on the dual~\eqref{eqn:main-dual}, the linear convergence rate parameter~\cite{nesterov2018lectures, karimi2016linear} would be $(1-\mu/\sigma)$.
The parameter in \Cref{thm:outer-conv} has an additional factor $1/\alpha_{\{\ell\}}$,
which arises due to the use of an active set method in conjunction with the augmentation rules. We present an example in \Cref{subsubsec:unavoidable} showing that this factor is perhaps unavoidable. 
\end{remark}

\begin{remark}\label{remark-inexact-compute}
\hyperlink{algo}{Algorithm~1} allows for both exact and inexact optimization of the reduced problem---the guarantees for \Cref{thm:outer-conv} apply to both variants. {{In particular, if we use a fixed number of PGD iterations for every sub-problem, we would obtain a linear convergence rate on the total number of inner iterations.}} 
Furthermore, \Cref{thm:outer-conv} applies to both deterministic and randomized augmentation rules.

Although our theory is based on a proximal gradient update on the reduced problem, the theory also applies to other descent algorithms that have a sufficient decrease condition, e.g. accelerated proximal gradient (APG) methods~\cite{beck2009fast,nesterov2018lectures} and L-BFGS~\cite{liu1989limited} with some modifications--see discussions in \Cref{remark:FISTA-LBFGS}.
\end{remark}

\begin{remark}\label{remark:rates}
\Cref{thm:outer-conv} implies that we need {at most} $O\left(({\sigma\alpha_{\{\ell\}}}/{\mu})\log({1}/{\epsilon})\right)$-many outer iterations,
to achieve an accuracy level of $\epsilon>0$. This
quantifies the worst-case convergence rate via $\alpha_{\{\ell\}}$.  We note that the constant $\sigma\alpha_{\{\ell\}}/\mu$ can be large when $n$ is large, making the speed of linear convergence slow. However,
by following the arguments in the proof, 
one can consider a more optimistic upper bound on the number of iterations,
by replacing $\alpha_{\{\ell\}}$ with $\beta_{\{\ell\}}$ (cf Lemma~\ref{lem:outer-conv-lem2-main}), as we discuss below.
\end{remark}

\noindent {\bf Understanding the costs of different rules:}
To better understand the computational costs of the different rules, we 
consider a setting where the maximum number of selected pairs (i.e., $|\Delta_{\{\ell\}}'|$) is set to be $O(n)$ across all rules\footnote{The choice of $P,K,M,G$ are specified in the caption of \Cref{table:comparison}}. 

According to \Cref{remark:rates}, to achieve $\epsilon$ accuracy, we need 
$O( (\sigma\alpha_{\{\ell\}}/{\mu})\log({1}/{\epsilon}))$ outer iterations in the worst case and $O(({\sigma\beta_{\{\ell\}}}/{\mu})\log({1}/{\epsilon}))$ outer iterations in the best case. Hence, for a fixed $\epsilon>0$, the total augmentation cost is proportional to $\alpha_{\{\ell\}}\times$\textit{Augmentation Cost} for the worst case and $\beta_{\{\ell\}}\times$\textit{Augmentation Cost} for the best case. See  \Cref{table:comparison} for an illustration. \begin{table}[h]\centering
\scalebox{0.95}{\begin{tabular}{c|cccc}
     \hline
     Rule& $|\Delta'_{\{\ell\}}|$&Augmentation Cost&$ \alpha_{\{\ell\}}$ &  $\beta_{\{\ell\}}$\\
     \hline
     1& $O(n)$&$O(n^2d+d|W|)$&$O(n)$&$O(1)$\\
     2&$O(n)$&$O(nd+d|W|)$&$O(n)$&$O(n)$\\
     3& $O(n)$&$O(nd+d|W|)$&$O(n)$&$O(n)$\\
     4&$O(n)$&$O(n\sqrt{n}d+d|W|)$&$O(n)$&$O(\sqrt{n})$\\
     5&$O(n)$&$O(n\sqrt{n}d+d|W|)$&$O(n)$&$O(\sqrt{n})$\\
     \hline
\end{tabular}}
\caption{\small {Comparing Augmentation Rules: We present an instance of \Cref{table:summary} where $|\Delta_{\{\ell\}}'|$ is the same across $\ell$. Specifically, we set $P=1$ for Rule 1, $K=n$ for Rule 2, $P=1$ for Rule 3, $M=n\sqrt{n},K=n$ for Rule 4, and $G=\sqrt{n},P=\sqrt{n}$ for Rule 5. Here, we ignore log-terms in the big $O$ notation.}}
\label{table:comparison}
\end{table}

We make a note of the following observations from \Cref{table:comparison}. 
\begin{compactitem}
\item As the maximum size of the augmentation set (i.e, $|\Delta_{\{\ell\}}'|$) at each iteration is the same across all rules, $\alpha_{\{\ell\}}$ is the same across all the five rules $\ell$. 

\item  In the worst case, different rules take the same number of iterations to achieve $\epsilon$ accuracy. 
The largest total augmentation costs (i.e., $\alpha_{\{\ell\}}\times$Augmentation Cost) are incurred by the purely greedy rule (Rule 1), and then the random-then-greedy rules (Rule 4 and Rule 5).

\item In the best case, the greedy rules take fewer iterations---the purely greedy rule (Rule~1) is better than random-then-greedy rules (Rules 4 and 5). The purely random rules (Rules 2 and 3) have similar iteration counts for the best and worst cases (as $\alpha_{\{\ell\}} = \beta_{\{\ell\}}$). 
\end{compactitem}

\medskip

The above discussion pertains to our theoretical guarantees.
{We note that the parameter ${\mu}/{(\sigma\alpha_{\{\ell\}})}$
appearing in Theorem~\ref{thm:outer-conv} can be difficult to compute (see Section~\ref{app:lemma-5-helper-thm} for a description of $\mu$), and may be small for large datasets making the linear rate slow in the worst-case scenario.
Our numerical experiments offer a refined understanding of the operating characteristics of the different algorithms, and what occurs in practice.}

{We note that though our work focuses on a problem arising from  convex regression, our algorithmic framework with randomized and random-then-greedy augmentation rules can also be applied to other convex quadratic programming problems involving a large number of decision variables with nonnegativity constraints.}

\subsubsection{Related work}
As discussed earlier, \hyperlink{algo}{Algorithm~1} (with exact optimization for the subproblems) is inspired by constraint generation, a classic tool for solving large scale linear programs~\cite{dantzig1960decomposition,bertsimas1997introduction}. Recently~\cite{bertsimas2021sparse} explore constraint generation for convex regression, but our framework differs. 
\cite{bertsimas2021sparse} use a primal approach for~\eqref{eqn:nonreg} and solve the restricted primal problem to {\emph{optimality}} using commercial QP solvers (e.g., Gurobi). They do not present computational guarantees for their procedure. 
In this paper we consider the subgradient regularized problem~\eqref{eqn:main-primal}, and our algorithms are on the {\emph{dual}} problem~\eqref{eqn:main-dual}. We extend the framework of~\cite{bertsimas2021sparse} to allow for {\emph{both}} exact and inexact optimization of the reduced problems. As a main computational bottleneck in the constraint generation procedure is in the augmentation step which requires scanning all the $O(n^2)$ constraints, we present randomized augmentation rules, which examine only a small subset of the constraints, before augmenting the active-set. Our computational guarantees, which to our knowledge are novel, apply to both the exact and inexact sub-problem optimization processes and deterministic/randomized augmentation rules.

As far as we can tell, with the exception of~\cite{bertsimas2021sparse}, convex optimization-based approaches for convex regression that precede our work (e.g,~\cite{mazumder2018computational,lin2020efficient,aybat2016parallelizable}), do not use active set-type methods and hence apply to smaller problem instances $n\approx 3000$. We note that active-set type methods are commonly used for convex QPs see, e.g., \cite{laiu2019constraint,friedman2010regularization} and references therein. We also explore active-set type methods, (randomized) augmentation rules, inexact optimization of sub-problems, and their computational guarantees for subgradient regularized convex regression.

The primal problem~\eqref{eqn:main-primal} can also be viewed as a projection problem onto a polyhedron with $O(n^2)$ constraints. \cite{hager2016projection} consider the problem of projecting a point into a polyhedron. To solve this problem, they 
explore some active set ideas and develop convergence guarantees for their procedure. The active-set subproblem considered in~\cite{hager2016projection} is different from what we consider.
Furthermore, our work differs in using randomized augmentation rules, inexact optimization of the subproblems, and their associated computational guarantees. Using problem structure, we can address instances larger than those studied in~\cite{hager2016projection}.

\section{Primal Feasibility and Duality Gap}\label{sec:dualgap}
Given a feasible solution for the full dual~\eqref{eqn:main-dual}, 
{{we show how it can be used to obtain a feasible solution for the primal problem~\eqref{eqn:main-primal} while also achieving a good primal objective value}}
(note that \eqref{def:primal-dual} need {\it{not}} deliver a primal feasible solution).
A primal feasible solution is useful as it leads to a convex function estimate and also a duality gap certificate.

Given a primal candidate $(\bm{\phi},\bm{\xi})$, we will
 construct $(\tilde{\bm{\phi}},\tilde{\bm{\xi}})$ that is feasible for~\eqref{eqn:main-primal}. 
To this end, let $v_{(i,j)}=\phi_j-\phi_i-\langle \bm{x}_j-\bm{x}_i,\bm{\xi}_i\rangle,$ for all $i\neq j$ and we use the convention $v_{(j,j)}=0$. Furthermore, let
 \begin{equation}
 \nu_j= \min_i~v_{(i,j)} ~~~\text{and}~~~\kappa_j\in\arg\min\{\lVert\bm{\xi}_k\rVert:k\in \arg\min_iv_{(i,j)}\}
 \end{equation}
and define 
$(\tilde{\bm{\phi}},\tilde{\bm{\xi}}):=\M{\psi}(\M\phi, \M\xi)$ as follows
 \begin{equation}\label{eqn:efn-feasible}
\tilde{\bm{\xi}}_j=\bm{\xi}_{\kappa_j}~~~\text{and}~~~~\tilde{\bm{\phi}}=\bm{\phi}-\bm{\nu}+c\bm{1} 
\end{equation}
where,  $c$ is such that $\bm{1}^\top\tilde{\bm{\phi}}=\bm{1}^\top\bm{y}$.
The following proposition shows that $(\tilde{\bm{\phi}},\tilde{\bm{\xi}})$ is feasible for the full primal problem.
\begin{proposition}\label{prop:primal-feas-soln}
For any pair $(\bm{\phi},\bm{\xi})$, the pair $(\tilde{\bm{\phi}},\tilde{\bm{\xi}})$ defined in~\eqref{eqn:efn-feasible},
is feasible for~\eqref{eqn:main-primal}. 
Furthermore, if $(\bm{\phi}^\star,\bm{\xi}^\star)$ is an optimal solution for~\eqref{eqn:main-primal} then it is a fixed point for the map $\M{\psi}(\cdot, \cdot)$, i.e., $\M\psi(\bm{\phi}^\star,\bm{\xi}^\star)=(\bm{\phi}^\star,\bm{\xi}^\star).$ 
\begin{proof}[{Proof of \Cref{prop:primal-feas-soln}}]
Given $\{{\phi}_{i}\}$, $\{\bm{\xi}_{i}\}$ and a scalar $c$, we define the following piecewise linear convex function (in $\M{x}$):
$$\varphi(\bm{x})=\max_{i \in [n]}~\{\langle \bm{x},\bm{\xi}_i\rangle +(\phi_i-\langle \bm{x}_i,\bm{\xi}_i\rangle)\}+c.$$
For any $j$,  we have
$$\max_{i\in[n]}\{\langle \bm x_j,\bm\xi_i\rangle+(\phi_i-\langle \bm x_i,\bm\xi_i\rangle)\}=\max_{i\in[n]}\{\phi_j-(\phi_j-\phi_i-\langle\bm x_j-\bm x_i,\bm\xi_i\rangle)\}=\phi_j-\min_{i \in [n]} v_{(i,j)}.$$
By definition of $\tilde{\bm\phi},\tilde{\bm\xi}$ given by~\eqref{eqn:efn-feasible}, it follows that
$$\varphi(\bm{x}_j)=\tilde{\phi}_j\quad\text{and}\quad \tilde{\bm\xi}_{j} =\bm{\xi}_{\kappa_j}\in\partial \varphi(\bm{x}_j).$$
Therefore, $(\tilde{\bm{\phi}},\tilde{\bm{\xi}})$ is feasible for the full primal problem.

Due to the feasibility of $(\bm{\phi}^\star,\bm{\xi}^\star)$, we have $\bm{v}^\star=\bm{A}\bm{\phi}^\star+\bm{B}\bm{\xi}^\star\geq \M{0}$, and thus $\min_iv_{(i,j)}=0=v_{(j,j)}$. This means $\bm\nu = \M{0}$, so it suffices to show that $c=0$. Note that the first term in the primal objective $\frac{1}{2}\norm{\bm\phi-\bm y}^2$ is minimized when $\bm{\phi}$ and $\bm y$ have the same mean (otherwise we can add a constant to $\bm\phi$ to decrease the objective). Therefore, the $c$ that satisfies $\bm{1}^\top(\bm\phi^\star+c\bm{1})=\bm{1}^\top\bm y$ must be 0. Thus, we have shown that $\tilde{\bm{\phi}}^\star=\bm{\phi}^\star$.
If $\tilde{\bm{\xi}}^\star\neq\bm{\xi}^\star$, then there exists $j$ such that $\kappa_j\neq j$, which means $\lVert\bm{\xi}_{\kappa_j}^\star\rVert\leq \lVert\bm{\xi}_j^\star\rVert$, and we have $f(\tilde{\bm{\phi}}^\star,\tilde{\bm{\xi}}^\star)\leq f(\bm{\phi}^\star,\bm{\xi}^\star)$. This contradicts the strict minimality of $(\bm{\phi}^\star,\bm{\xi}^\star)$ as the primal objective function $f$ is strictly convex.
\end{proof}
\end{proposition}

As discussed in Section~\ref{subsec:active-set-aug}, each $(\phi_i,\bm\xi_i)$ defines a hyperplane $H_i:y=\langle \bm x-\bm x_i,\bm\xi_i\rangle+\phi_i $ containing $P_i=(\bm x_i,\phi_i)$. However, $(\bm\phi,\bm\xi)$ may not satisfy the convexity constraint (i.e., $P_j$ may not lie above $H_i$); and $|\nu_j|$ quantifies how far $P_j$ lies below the piecewise maximum of $H_i$s. Intuitively, $(\tilde{\bm\phi},\tilde{\bm\xi})$ attains feasibility by taking a piecewise maximum of these hyperplanes and then adjusts itself by a constant in an attempt to decrease the primal objective.

\medskip

\noindent {\bf Duality Gap:} In light of strong duality between \eqref{eqn:main-primal} and \eqref{eqn:main-dual}, we have 
$L(\bm{\lambda}^\star)=-f(\bm{\phi}^\star,\bm{\xi}^\star)$. 

From a dual feasible solution $\B\lambda$, we can obtain a primal variable  $(\M\phi, \M\xi)$ via~\eqref{def:primal-dual}. If $\M\psi(\M\phi, \M\xi)$ is a primal {\emph{feasible}} solution obtained from $(\M\phi, \M\xi)$ by the operation defined in~\eqref{eqn:efn-feasible}, then we can compute a duality gap via:
$$\underline{L}(\bm{\lambda})=-f(\M\psi(\M\phi, \M\xi)) \leq L^\star\leq L(\bm{\lambda}).$$
The gap $L(\bm{\lambda})-\underline{L}(\bm{\lambda})$ equals zero if and only if $\B{\lambda}$ is an optimal solution for \eqref{eqn:main-dual}.

\section{Numerical Experiments}\label{sec:experiments}
We present numerical experiments to study the different variants of our algorithm; and compare it with current approaches. As our focus is on large-scale instances of~\eqref{eqn:origin}, we study both real and synthetic datasets in the range $n \sim 10^4$-$10^5$ and $4\leq d\leq 20$.

\medskip

\noindent {\bf Datasets:} We consider the following synthetic and real datasets for the experiments.

\smallskip

\noindent {\it Synthetic Data:} 
Following~\cite{mazumder2018computational}, we generate data 
via the model $y_i = \phi^0(\bm{x}_i)+\epsilon_i, i\in[n]$ where, 
 $\phi^0(\B{x})$ is a convex function, 
$\epsilon_{i} \stackrel{\text{iid}}{{\sim}} N(0, \gamma^2)$ and $\gamma$ is chosen to match a specified value of
signal-to-noise ratio, $\mathrm{SNR}=\| \B\phi^0\|^2/ \|\B\epsilon\|^2=3$.
The covariates are drawn independently from a uniform distribution on $[-1,1]^d$.
Every feature is normalized to have zero mean and unit $\ell_2$-norm; and $\B{y}$ has zero mean and unit $\ell_{2}$-norm. We consider the following choices of the underlying true convex function $\phi^0(\B{x})$.

\smallskip

\begin{compactitem}
    \item SD1: $\phi^0(\bm{x})=\norm{\bm{x}}_2^2$
    \item SD2: $\phi^0(\bm{x})=\max_{1\leq i\leq 2d}\{\bm{\xi}_i^\top \bm{x}\}$, where $\bm{\xi}_i\in\mathbb{R}^d$ are independently drawn from the uniform distribution $[-1,1]^{d}$.

\end{compactitem}

\medskip

\noindent {\it Real Data:} 
We also consider the following real datasets in our experiments:

{{\begin{compactitem}
        \item RD1, RD2: These two datasets, studied in~\cite{kaya2019predicting} have $n=10,000$ and $d=4$. 
        \item RD3, RD4: These two datasets, studied in~\cite{zhang2017cautionary} have $n=10,000$ and $d=4$.
        \item RD5: This dataset, studied in~\cite{liang2015assessing}, has $n=10,000$ and $d=4$.
        \item RD6: This dataset, studied in~\cite{kaya2012local,tufekci2014prediction}, has $n=5,000$ and $d=4$.
        \item RD7: This dataset, from earlier work~\cite{mekaroonreung2012estimating,mazumder2018computational,bertsimas2021sparse} has $n=30,000$ and $d=4$. 
        \item RD8: This dataset, from~\cite{ramsey2012statistical,hannah2013multivariate,balazs2016convex}, has $n=10,000$ and $d=4$.
\end{compactitem}}}
For additional details on the real datasets see \Cref{app:real-data-details}. In all above datasets, covariates and response are centered and scaled so that each variable has unit $\ell_2$-norm.

\medskip
\noindent {\bf Algorithms:}  
We compare our approach versus the cutting plane (CP) method~\cite{bertsimas2021sparse} {and the ADMM method~\cite{mazumder2018computational} using the authors' implementations. All algorithms are run with the time limit of 12 hours}\footnote{{For the cutting plane method which uses Gurobi, it can take a long time to solve the reduced sub-problem thereby exceeding the allocated 12 hr time-limit before obtaining an accurate solution.}}. 
{For our algorithms, we consider two-stage methods:
In the first stage, we use random rules 2/4 for augmentation and perform inexact optimization over the active set---empirically, this results in good initial progress in the objective value but then the progress slows down as the augmentation rules include very few violated constraints. When the number of added constraints is less than $0.005n$ for consecutive 5 iterations\footnote{This is a choice we used in our experiments, and can be tuned in general for performance benefits.} we switch to the second stage where we use occasional greedy rules 1/5 for augmentation and exact optimization of sub-problems.}
Note that our theory (Section~\ref{sec:algo}) guarantees convergence of this two-stage procedure. 
Additional details on algorithm parameter choices can be found in \Cref{subsec:algo-params}. {In this section, we will use Rule~$a$-$b$ ($a\in\{2,4\},b\in\{1,5\}$) to denote different variants of our algorithm with random rule $a$ (Stage~1) and greedy rule $b$ (Stage~2).}

\medskip

\noindent {\bf Software Specifications:} 
All computations were carried out on MIT's Engaging Cluster on an Intel Xeon 2.30GHz machine, with one CPU and 8GB of RAM.
{{For ADMM and CP, we used a larger amount of memory 64GB RAM.}} Our algorithms are written in Julia (v1.5), and 
our code is available on github at:

~~~~~~~~~~~~~~~~~~~\url{https://github.com/wenyuC94/ConvexRegression}

\smallskip

\noindent\textbf{Performance of proposed algorithms: } \Cref{fig:profile} presents the relative objective of the dual, defined as
 $$\text{Rel. Obj.} = {(L(\bm\lambda^t)-L^\star)}/(|L^\star|+1)$$
as well as the primal infeasibility (\texttt{pinfeas})---also used in \cite{mazumder2018computational,bertsimas2021sparse}---defined as\footnote{{For CP~\cite{bertsimas2021sparse} and ADMM~\cite{mazumder2018computational}, \text{Rel. Obj.} is defined as $|f(\bm\phi^t,\bm\xi^t)-f^\star|/(1+|f^\star|)$, where $f^\star=-L^\star$; \texttt{pinfeas} is defined as $\|(\bm A\bm\phi^t+\bm B\bm\xi^t)_-\|/n$. Here $\bm a_+$ and $\bm a_-$ denotes the vector of $[\max\{a_j,0\}]_{j}$ and $[\min\{a_j,0\}]_j$, respectively.}}
$$\texttt{pinfeas} = \frac1n\|(\nabla L(\bm\lambda^t))_+\|$$ 
for different algorithms versus time (in seconds).
Above, $L^\star$ is taken as the minimum objective among solutions obtained by all the algorithms running for {12 hours}.
\Cref{fig:profile} considers synthetic datasets SD1 and SD2 with different $n$'s and $d$'s, and the real dataset RD1. {We note that the choices of $\rho$ displayed in \Cref{fig:profile} correspond to good statistical performance---this is discussed in further detail below and in \Cref{fig:rmse}. \Cref{fig:profile} suggests that our algorithms perform better than CP and ADMM -- both in terms of Rel. Obj. and primal infeasibility. 
Note that ADMM~\cite{mazumder2018computational} does not use active-set methods and consumes prohibitively large memory when $n\geq 30,000$.}

{\Cref{table:runtime} presents a snapshot of the runtimes of our proposed algorithms, CP and ADMM.} Here, the runtime corresponds to the time (s) taken by an algorithm to achieve a $0.05$ relative objective. Our proposed two-stage algorithms can achieve this accuracy quite quickly, while CP and ADMM are often unable to converge to that accuracy. We observe that for our algorithms, runtime generally increases with smaller $\rho$-values. 
\Cref{table:runtime-largen} presents results for larger datasets with $n=100,000$---we show two of our methods and do not include greedy augmentation Rule~1 due to large computational costs. As expected, for such large problems randomized augmentation rules play a key role, and our method is better at obtaining solutions with low to moderate level of precision.

\begin{figure}[h!]
    \centering
\resizebox{0.98\textwidth}{!}{\begin{tabular}{r c r c}
\multicolumn{4}{c}{ {Objective Profiles for SD1, $n=10,000$, $d=4$ and $\rho=10^{-4}$}}\\
\rotatebox{90}{  { {~~~~~~~~~~~~~Rel. Obj.}}}&\includegraphics[width =0.48 \textwidth,trim = 1cm 0cm 2.25cm 2cm,clip ]{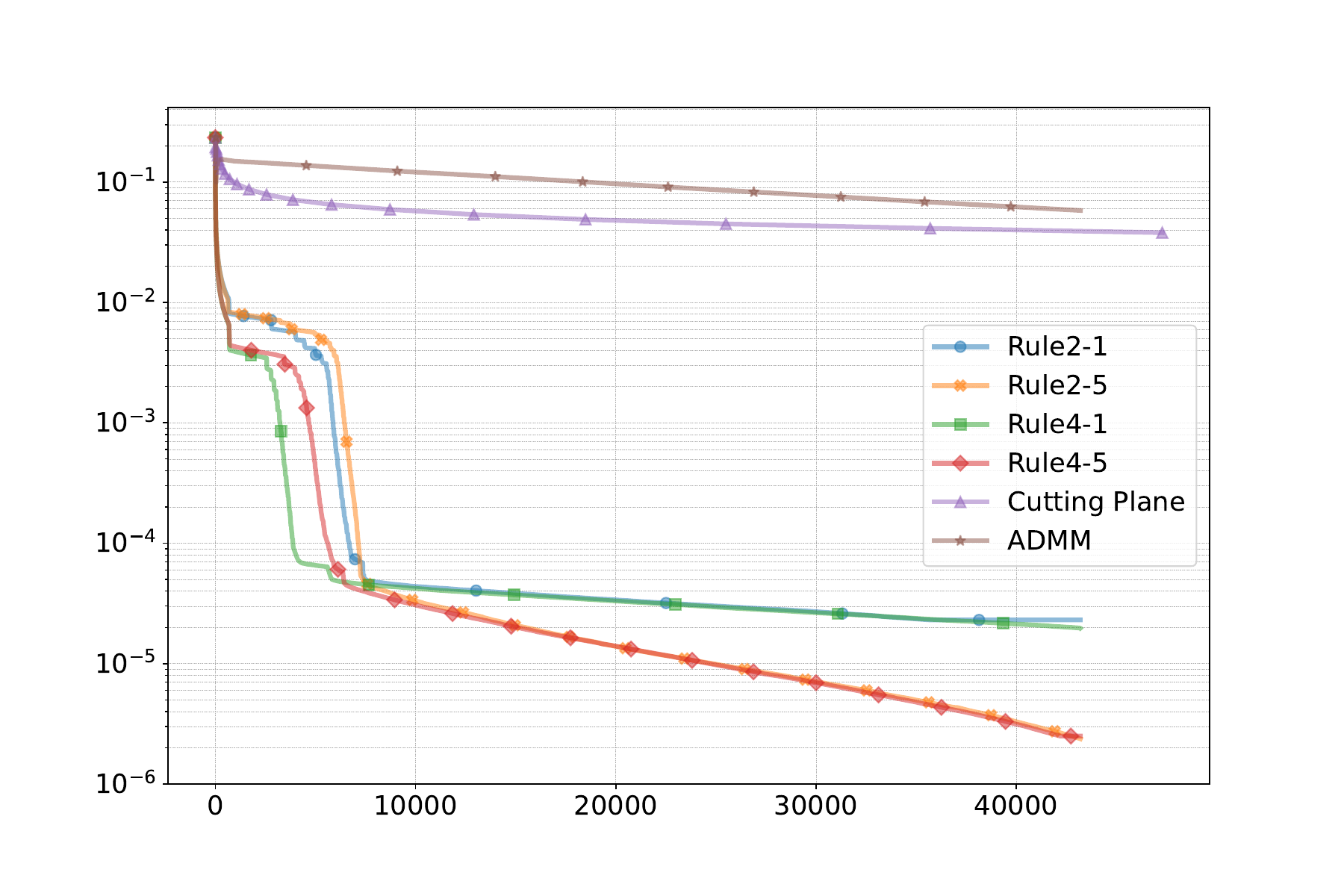}&\rotatebox{90}{  { {~~~~~~~~~~~~~\texttt{pinfeas}}}}&\includegraphics[width =0.48 \textwidth,trim = 1cm 0cm 2.25cm 2cm,clip ]{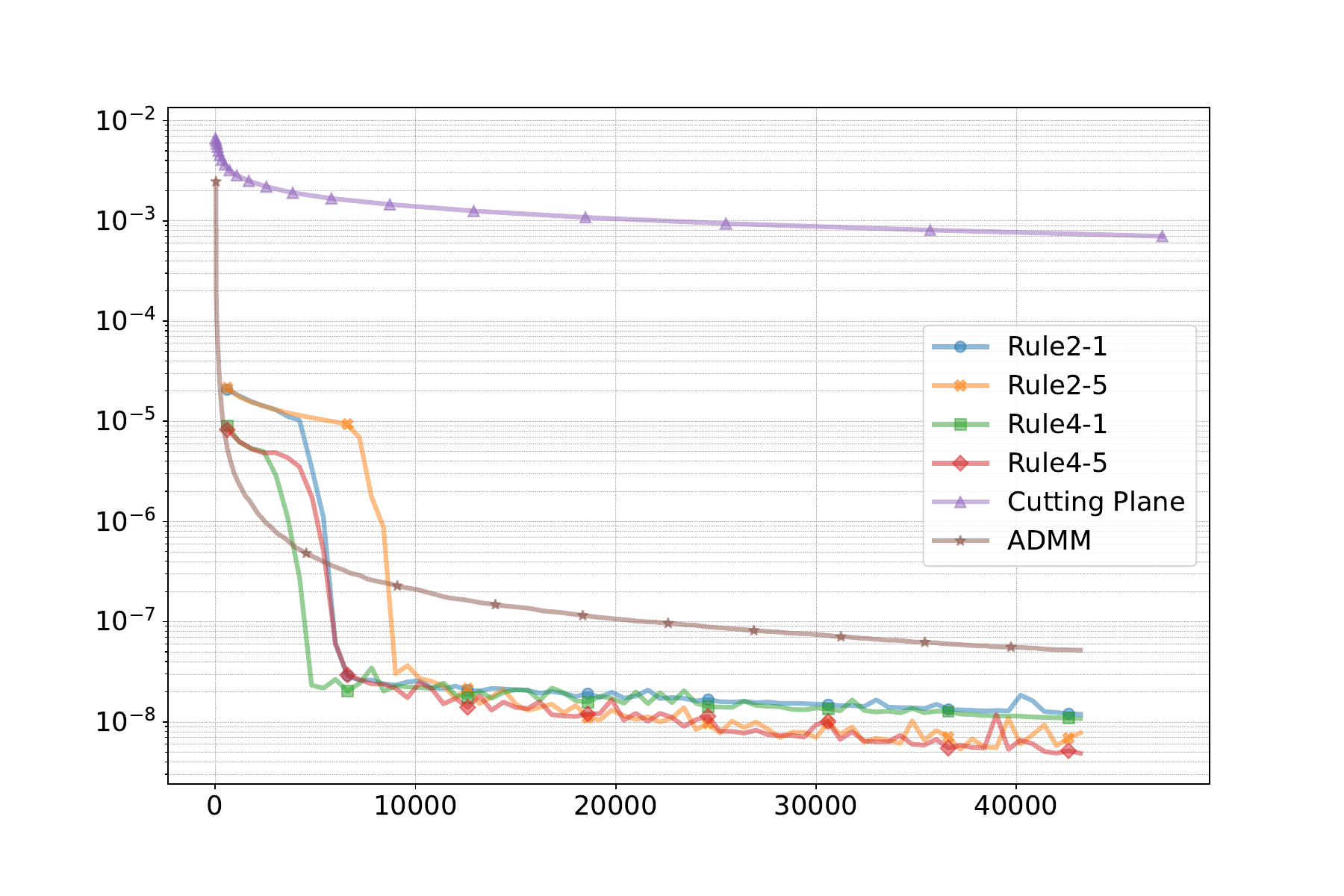}\\
  \multicolumn{4}{c}{ {Objective Profiles for SD1, $n=30,000$, $d=4$ and $\rho=10^{-4}$}} \\
\rotatebox{90}{  { {~~~~~~~~~~~~~Rel. Obj.}}}&\includegraphics[width =0.48 \textwidth,trim = 1cm 0cm 2.25cm 2cm,clip ]{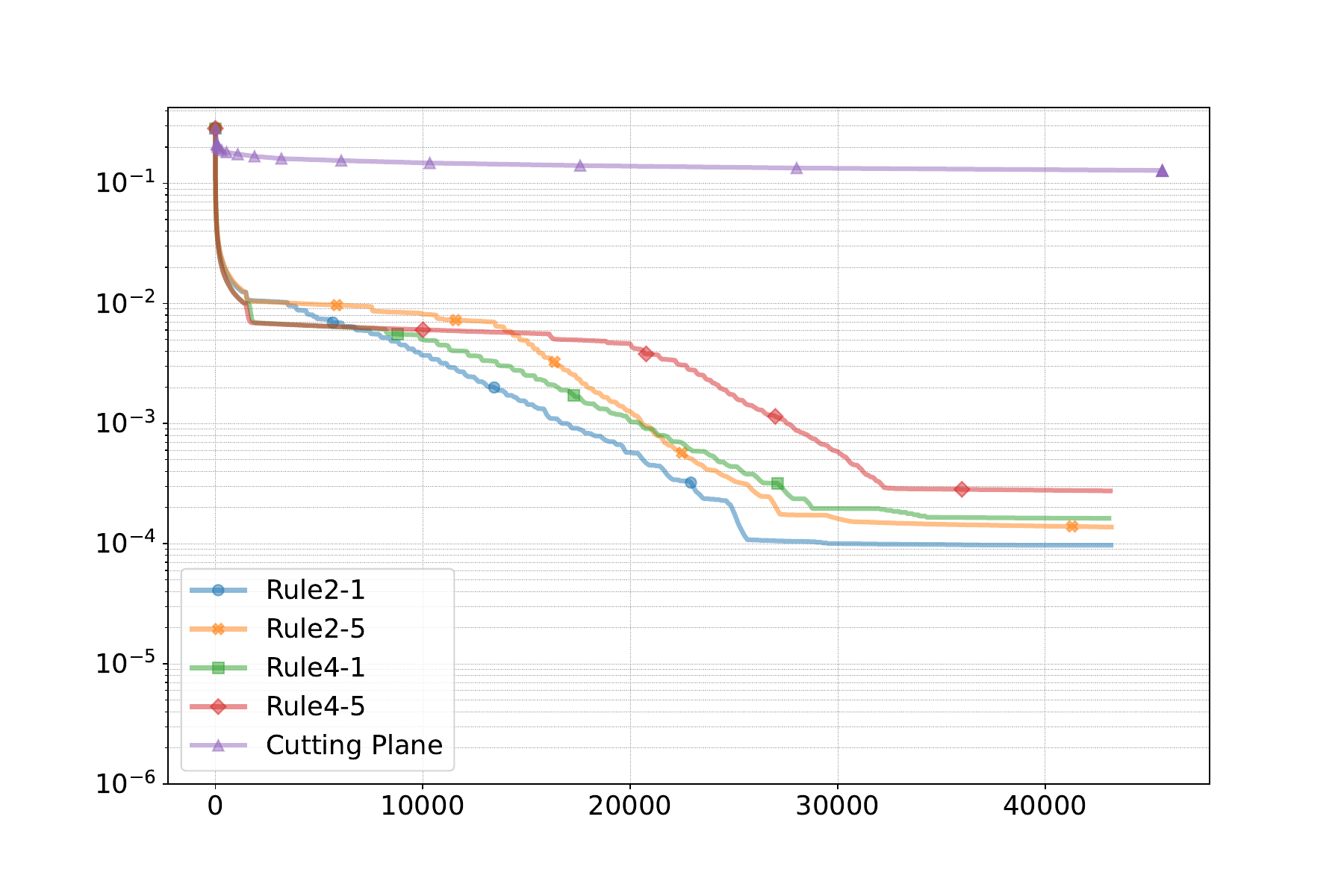}&\rotatebox{90}{  { {~~~~~~~~~~~~~\texttt{pinfeas}}}}&\includegraphics[width =0.48 \textwidth,trim = 1cm 0cm 2.25cm 2cm,clip ]{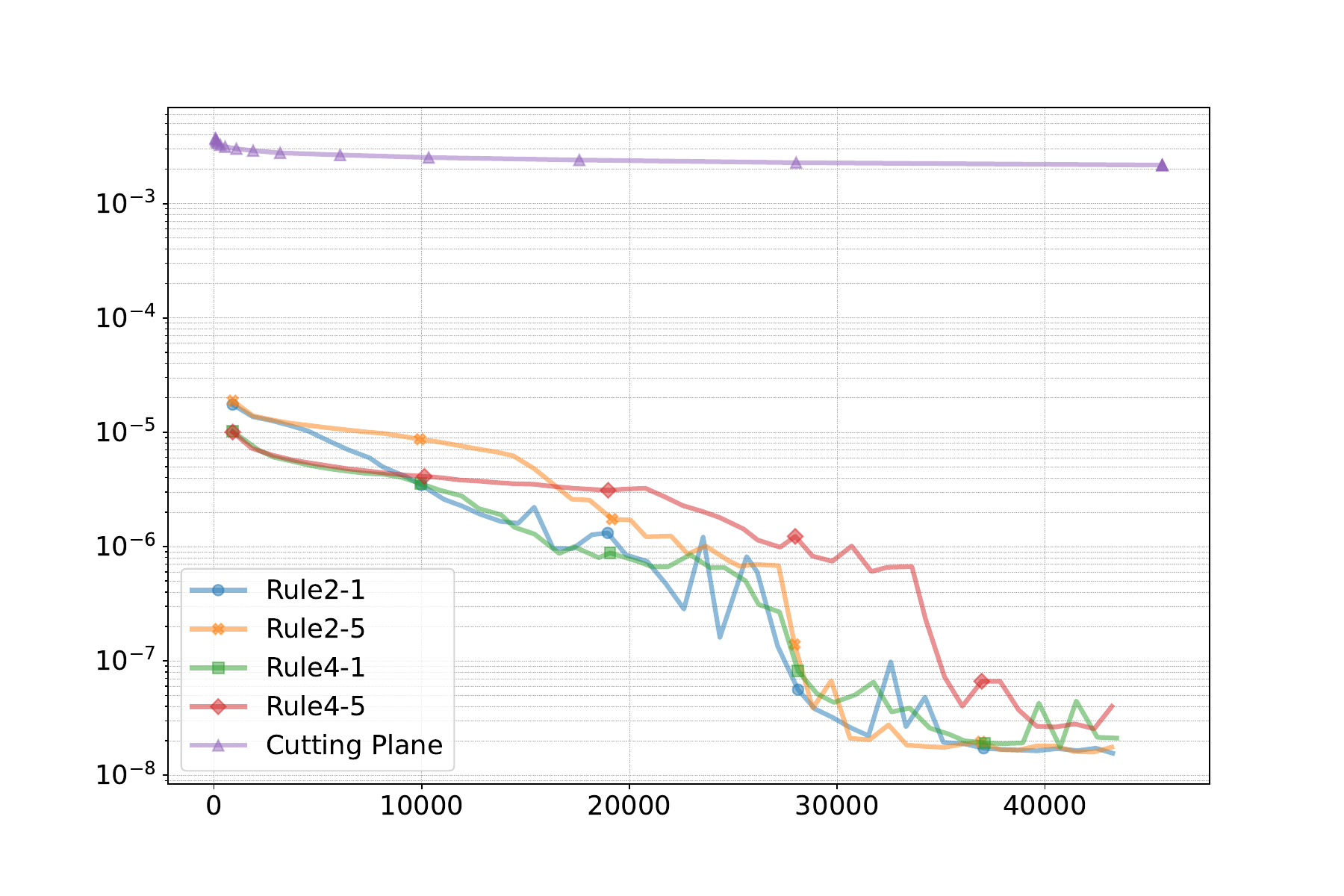}\\
 \multicolumn{4}{c}{ {Objective Profiles for SD2, $n=30,000$, $d=4$ and $\rho=10^{-4}$}} \\
\rotatebox{90}{  { {~~~~~~~~~~~~~Rel. Obj.}}}&\includegraphics[width =0.48 \textwidth,trim = 1cm 0cm 2.25cm 2cm,clip ]{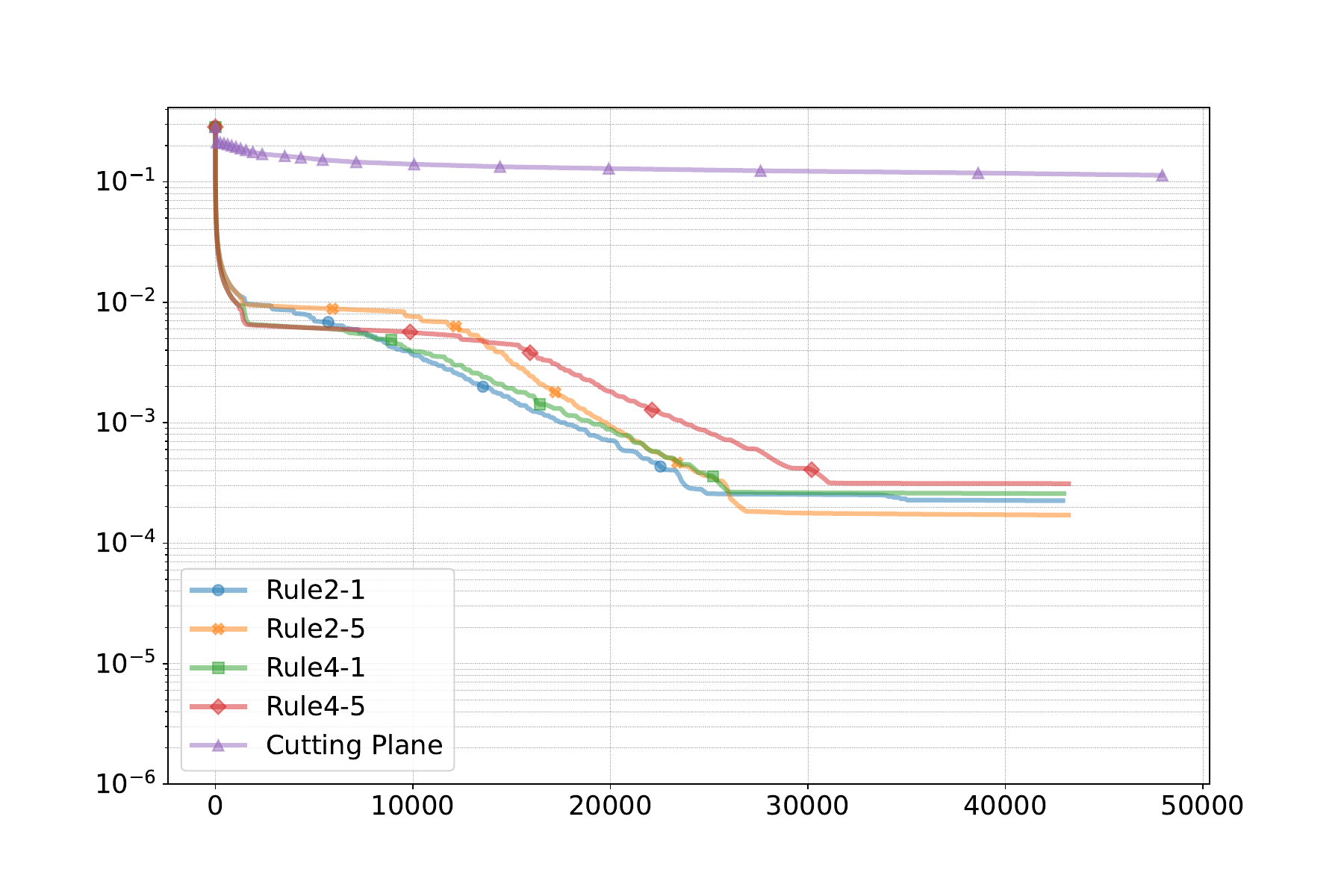}&\rotatebox{90}{  { {~~~~~~~~~~~~~\texttt{pinfeas}}}}&\includegraphics[width =0.48 \textwidth,trim = 1cm 0cm 2.25cm 2cm,clip ]{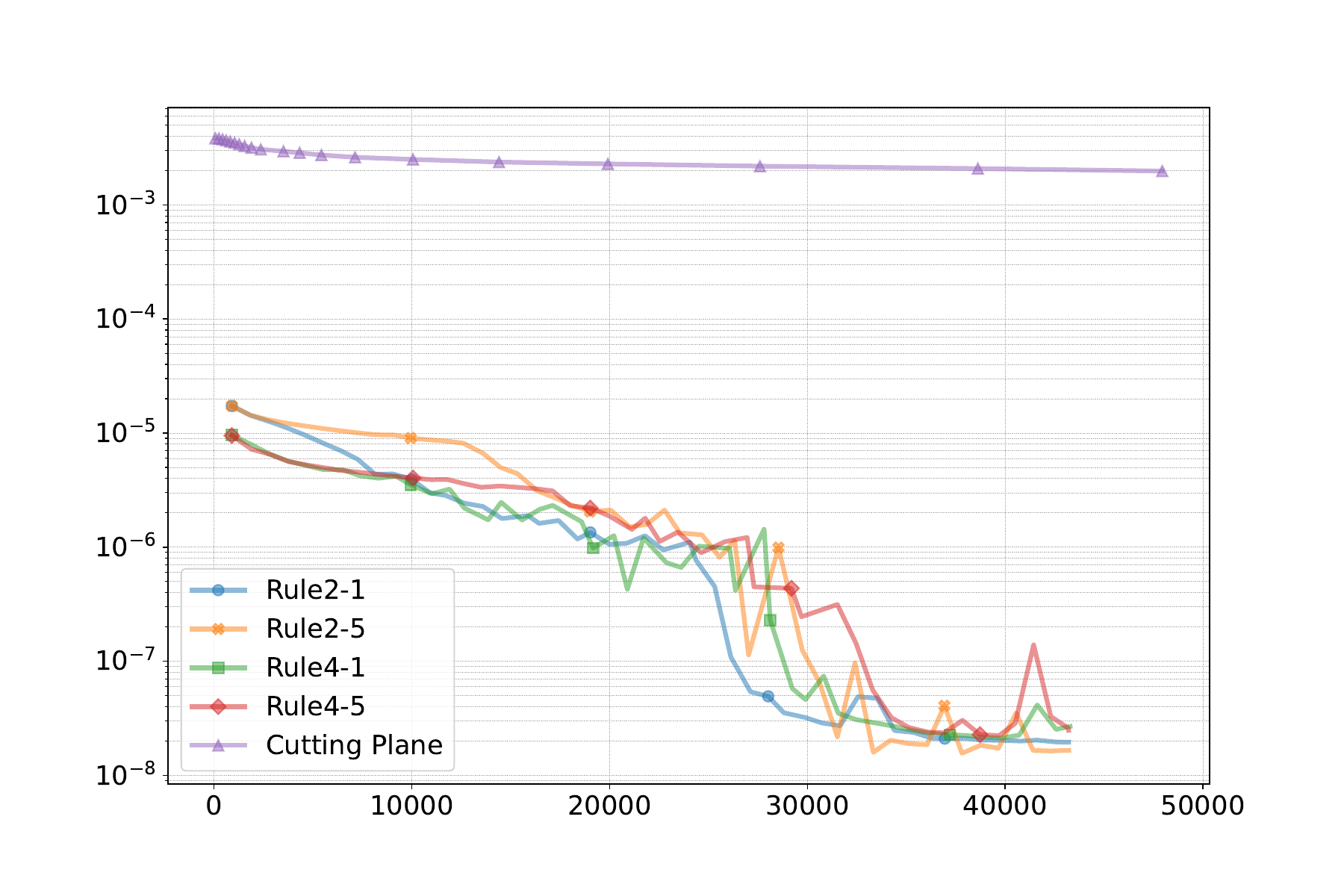}\\
\multicolumn{4}{c}{ {Objective Profiles for RD1, $n=10,000$, $d=4$ and $\rho=10^{-5}$}} \\
\rotatebox{90}{{{~~~~~~~~~~~~~~~~~Rel. Obj.}}}&\includegraphics[width =0.48 \textwidth,trim = 1cm 0cm 2.25cm 2cm,clip ]{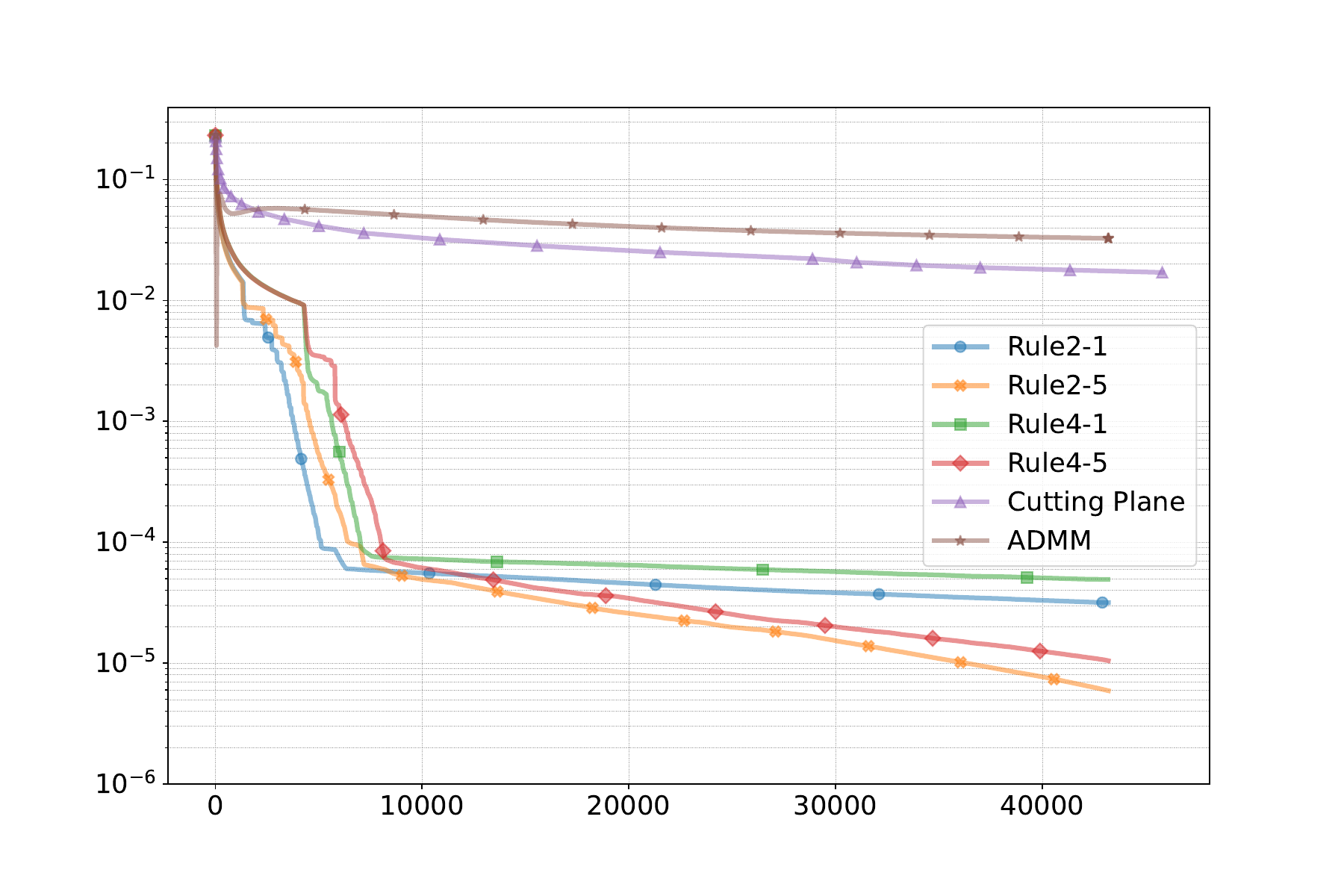}&\rotatebox{90}{  { {~~~~~~~~~~~~~\texttt{pinfeas}}}}&\includegraphics[width =0.48 \textwidth,trim = 1cm 0cm 2.25cm 2cm,clip ]{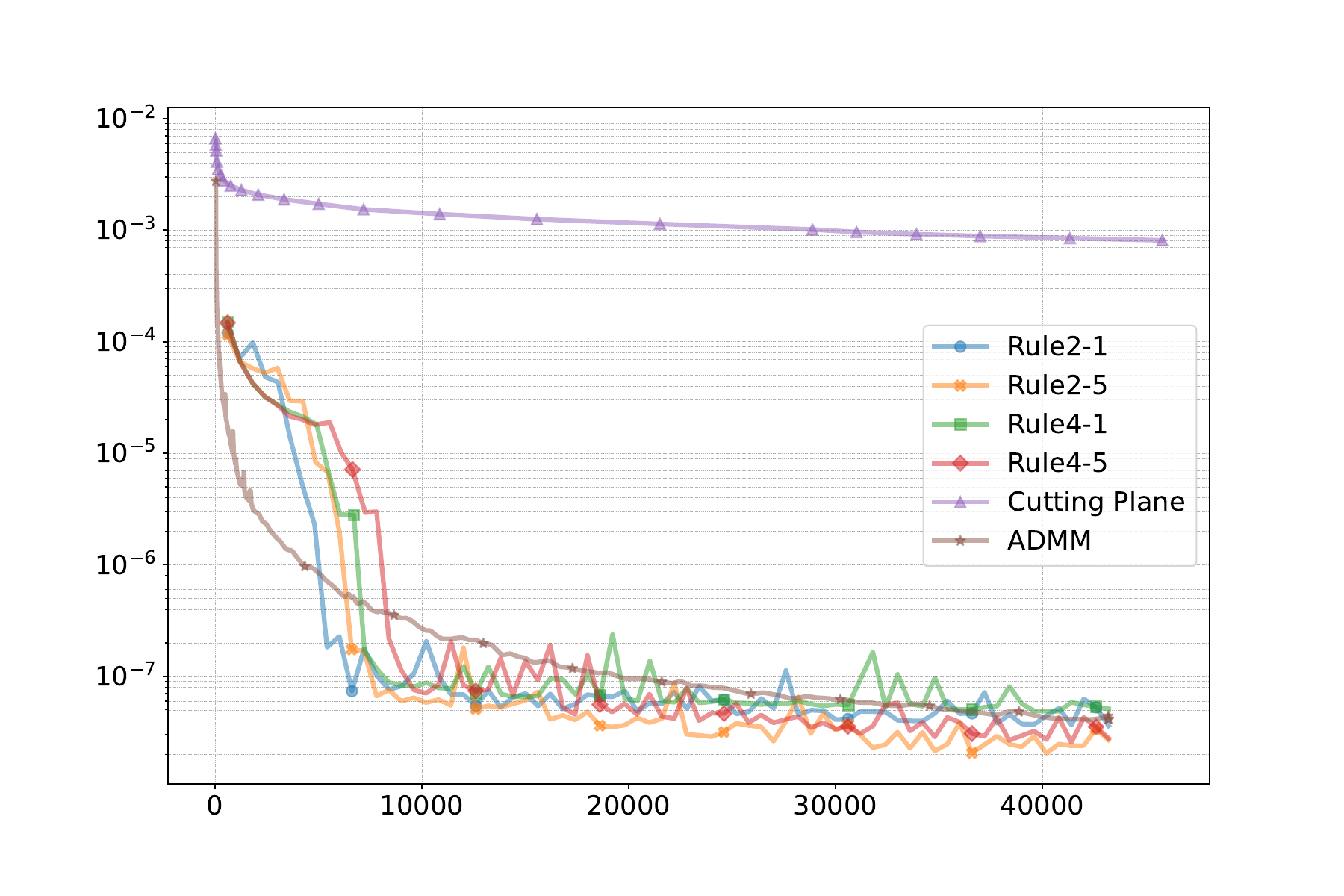}\\
& {\small{time (s)}} & & {\small{time (s)}} 
    \end{tabular}
    }
    \caption{{\small{Plots (in log-scale) of Relative Objective [left panel] and primal infeasibility [right panel]
   versus time (secs). We consider three synthetic data sets (top 3 rows) and a real dataset (bottom row). 
    We compare our algorithms against the cutting plane method~\cite{bertsimas2021sparse} and the ADMM method~\cite{mazumder2018computational}. For each algorithm, we run 5 repetitions, each bold line corresponds to the median of the profiles of one algorithm.
      The ADMM profiles (2nd, 3rd rows) are missing as they run out of memory (64GB).}}} 
    \label{fig:profile}
\end{figure}

\begin{table}[!ht]
\centering
\caption{{\small{Comparison of runtime (s) of our algorithms versus CP~\cite{bertsimas2021sparse} and ADMM~\cite{mazumder2018computational}. Here runtime refers to the time taken to achieve a Rel. Obj. of $5\times10^{-2}$. We report the median runtime and standard error (in bracket) across 5 replications (random instances). Note:
 `-' means that no replication of the algorithm achieves this level of relative accuracy within 4hrs, 
 `-*' means that some replications encountered convergence issues and others did not reach the tolerance within 4hrs,  `**' means all replications crash due to either numerical/memory problems. The entry 4320.90* (column=CP and 
 row=RD4) means that some of replications crashed due to numerical/memory issues, and we report the median runtime for the replications that did not crash.}}}
\label{table:runtime}
\resizebox{1.0\textwidth}{!}{
\begin{tabular}{lrrr|rrrrrr}
\toprule
data &      $\tfrac{n}{10^4}$ &    $d$ &  {\small{$\log_{10}\rho$}} &          Rule2-1 &          Rule2-5 &          Rule4-1 &          Rule4-5 &  CP~\cite{bertsimas2021sparse}  & ADMM~\cite{mazumder2018computational}\\
\midrule
 SD1 &  3 &   4 &      -3 &       3.81 {\footnotesize (0.95)} &       3.67 {\footnotesize (0.66)} &       2.33 {\footnotesize (0.25)} &       2.41 {\footnotesize (0.24)} &   -  & **\\
 SD1 &  3 &   4 &      -4 &      52.66 {\footnotesize (1.93)} &     52.35 {\footnotesize (35.44)} &      44.14 {\footnotesize (9.56)} &     43.71 {\footnotesize (45.87)} &   -  & **\\
 SD1 &  3 &  10 &      -3 &     31.56 {\footnotesize (10.62)} &     21.06 {\footnotesize (12.57)} &      10.64 {\footnotesize (8.69)} &      11.50 {\footnotesize (2.25)} &   - & **\\
 SD1 &  3 &  10 &      -4 &  1290.51 {\footnotesize (114.15)} &  1289.37 {\footnotesize (355.51)} &   364.57 {\footnotesize (202.95)} &   362.25 {\footnotesize (192.42)} &  -* & **\\
 SD1 &  3 &  20 &      -3 &    101.14 {\footnotesize (72.88)} &    180.22 {\footnotesize (64.26)} &     95.69 {\footnotesize (25.63)} &     93.91 {\footnotesize (28.92)} &   - & **\\
 SD1 &  3 &  20 &      -4 &  2767.88 {\footnotesize (373.64)} &  2774.59 {\footnotesize (416.69)} &   939.66 {\footnotesize (222.82)} &   898.22 {\footnotesize (297.24)} &   - & ** \\
 \midrule
 SD2 &  3 &   4 &      -3 &       4.29 {\footnotesize (3.22)} &       3.51 {\footnotesize (0.32)} &       2.33 {\footnotesize (0.14)} &       2.52 {\footnotesize (1.32)} &   - & **\\
 SD2 &  3 &   4 &      -4 &      39.22 {\footnotesize (1.83)} &     44.87 {\footnotesize (21.84)} &     37.05 {\footnotesize (21.41)} &     37.14 {\footnotesize (21.96)} &   -& ** \\
 SD2 &  3 &  10 &      -3 &     18.76 {\footnotesize (12.56)} &      19.00 {\footnotesize (3.86)} &      10.60 {\footnotesize (0.77)} &      14.00 {\footnotesize (8.59)} &   - & **\\
 SD2 &  3 &  10 &      -4 &   975.48 {\footnotesize (344.82)} &   978.28 {\footnotesize (330.26)} &   404.70 {\footnotesize (223.14)} &   403.52 {\footnotesize (217.71)} &  -*& ** \\
 SD2 &  3 &  20 &      -3 &     124.51 {\footnotesize (9.55)} &     124.22 {\footnotesize (9.39)} &     66.80 {\footnotesize (49.30)} &     66.01 {\footnotesize (49.05)} &   - & **\\
 SD2 &  3 &  20 &      -4 &  4292.47 {\footnotesize (365.84)} &  4693.06 {\footnotesize (434.83)} &  1400.07 {\footnotesize (249.20)} &  1187.88 {\footnotesize (119.68)} &   - & **\\
 \midrule
 RD1 &  1 &  4 &      -4 &    28.74 {\footnotesize (2.72)} &    28.68 {\footnotesize (2.81)} &    22.16 {\footnotesize (3.96)} &    21.72 {\footnotesize (4.12)} &   6827.02 {\footnotesize (963.13)} &                 - \\
 RD1 &  1 &  4 &      -5 &  168.99 {\footnotesize (13.11)} &  161.28 {\footnotesize (17.96)} &  276.40 {\footnotesize (20.29)} &  273.31 {\footnotesize (20.43)} &   3704.07 {\footnotesize (876.34)} & 9729.61 {\footnotesize (609.72)}\\
 RD2 &  1 &  4 &      -4 &    47.16 {\footnotesize (4.77)} &    61.88 {\footnotesize (3.99)} &    24.19 {\footnotesize (1.81)} &    23.02 {\footnotesize (2.65)} &                  - \\
 RD2 &  1 &  4 &      -5 &  264.68 {\footnotesize (21.88)} &  255.80 {\footnotesize (21.92)} &  115.46 {\footnotesize (43.57)} &  117.49 {\footnotesize (41.41)} &                 ** &     248.71 {\footnotesize (9.77)} \\
 RD3 &  1 &  4 &      -4 &    18.16 {\footnotesize (2.07)} &    18.01 {\footnotesize (2.06)} &    11.09 {\footnotesize (2.04)} &    11.31 {\footnotesize (2.09)} &                  - &                 - \\
 RD3 &  1 &  4 &      -5 &    83.41 {\footnotesize (4.48)} &    70.47 {\footnotesize (6.94)} &    75.07 {\footnotesize (6.94)} &    74.31 {\footnotesize (6.97)} &  7805.34 {\footnotesize (1325.51)} &     244.73 {\footnotesize (4.75)}  \\
 RD4 &  1 &  4 &      -4 &    17.59 {\footnotesize (0.76)} &    18.04 {\footnotesize (0.85)} &    17.72 {\footnotesize (1.08)} &    15.92 {\footnotesize (1.19)} &                  -   &                 -\\
 RD4 &  1 &  4 &      -5 &   128.86 {\footnotesize (4.95)} &   128.51 {\footnotesize (5.00)} &   107.64 {\footnotesize (6.75)} &   122.39 {\footnotesize (5.62)} &       4320.90* &     261.39 {\footnotesize (8.66)} \\
 RD5 &  1 &  4 &      -4 &    11.98 {\footnotesize (0.82)} &    11.76 {\footnotesize (0.84)} &     7.72 {\footnotesize (0.53)} &     7.70 {\footnotesize (0.55)} &                  - &                 -\\
 RD5 &  1 &  4 &      -5 &   104.33 {\footnotesize (7.92)} &   114.19 {\footnotesize (9.29)} &    83.22 {\footnotesize (5.31)} &    81.75 {\footnotesize (2.51)} &                  - &     248.79 {\footnotesize (7.82)}\\
 RD6 &  $\tfrac{1}{2}$ &  4 &      -4 &     3.59 {\footnotesize (0.22)} &     3.15 {\footnotesize (0.27)} &     2.61 {\footnotesize (0.25)} &     2.34 {\footnotesize (0.26)} &      233.47 {\footnotesize (6.98)} &  1707.11 {\footnotesize (115.08)} \\
 RD6 &   $\tfrac{1}{2}$ &  4 &      -5 &     9.35 {\footnotesize (0.68)} &    10.81 {\footnotesize (0.53)} &     7.41 {\footnotesize (1.05)} &     6.92 {\footnotesize (1.10)} &      64.06 {\footnotesize (13.64)} &      71.84 {\footnotesize (5.13)}  \\
 RD7 &  3 &  4 &      -4 &    20.84 {\footnotesize (3.08)} &    20.84 {\footnotesize (3.13)} &    17.15 {\footnotesize (3.35)} &    25.51 {\footnotesize (3.65)} &                  - &                ** \\
 RD7 &  3 &  4 &      -5 &    {40.81 {\footnotesize (7.80)} }&    { 41.31 {\footnotesize (7.73)}} &    {39.99 {\footnotesize (16.76)}} &   { 85.39 {\footnotesize (17.16)}} &                  - &                ** \\
 RD8 &  1 &  4 &      -4 &     4.87 {\footnotesize (0.12)} &     5.25 {\footnotesize (0.46)} &     4.98 {\footnotesize (0.85)} &     5.00 {\footnotesize (0.86)} &                  -  &                 -\\
 RD8 &  1 &  4 &      -5 &    37.56 {\footnotesize (4.07)} &    33.31 {\footnotesize (4.28)} &    54.85 {\footnotesize (8.00)} &    55.22 {\footnotesize (9.21)} &                  - &     235.58 {\footnotesize (5.29)}\\
\bottomrule
\end{tabular}
}
\end{table}

\begin{table}[!ht]
    \centering
      \scalebox{0.85}{
    \begin{tabular}{lrrr|rrr}
\toprule
data &      $n$ &    $d$ &  $\log_{10}\rho$ &                  Rule2-5 &                 Rule4-5 &  CP~\cite{bertsimas2021sparse}  \\
\midrule
 SD1 &  100000 &   4 &      -3 &      19.15 {\footnotesize (5.05)} &      16.50 {\footnotesize (1.12)} &   - \\
 SD1 &  100000 &   4 &      -4 &     94.52 {\footnotesize (21.63)} &    102.66 {\footnotesize (23.53)} &   - \\
 SD1 &  100000 &  10 &      -3 &      27.30 {\footnotesize (7.66)} &      22.48 {\footnotesize (6.07)} &   - \\
 SD1 &  100000 &  10 &      -4 &    598.83 {\footnotesize (54.45)} &    627.96 {\footnotesize (59.55)} &   - \\
 SD1 &  100000 &  20 &      -3 &     71.08 {\footnotesize (12.05)} &      38.84 {\footnotesize (6.29)} &   - \\
 SD1 &  100000 &  20 &      -4 &  2428.77 {\footnotesize (232.50)} &  1041.10 {\footnotesize (113.51)} &  -* \\
 \midrule
 SD2 &  100000 &   4 &      -3 &      29.89 {\footnotesize (2.62)} &      28.46 {\footnotesize (5.69)} &   - \\
 SD2 &  100000 &   4 &      -4 &    121.77 {\footnotesize (23.12)} &    111.20 {\footnotesize (19.85)} &   - \\
 SD2 &  100000 &  10 &      -3 &      27.08 {\footnotesize (2.60)} &      28.15 {\footnotesize (3.17)} &   - \\
 SD2 &  100000 &  10 &      -4 &    509.58 {\footnotesize (27.91)} &    422.83 {\footnotesize (73.77)} &   - \\
 SD2 &  100000 &  20 &      -3 &     75.59 {\footnotesize (13.97)} &      49.22 {\footnotesize (6.37)} &   - \\
 SD2 &  100000 &  20 &      -4 &  2810.69 {\footnotesize (262.19)} &  1397.47 {\footnotesize (134.45)} &   - \\
\bottomrule
\end{tabular}}
  \caption{{Runtime (s) of our algorithms, and the cutting plane (CP) method~\cite{bertsimas2021sparse} for $n=100,000$. ADMM runs out of memory (64GB) on these instances. See \Cref{table:runtime} for more details on the notations.}} \label{table:runtime-largen}
\end{table}

\noindent\textbf{Subgradient regularization can improve statistical performance:}  
As mentioned earlier, the presence of $\ell_2$-regularization on the subgradients (i.e., a value of $\rho>0$) in~\eqref{eqn:origin}, can lead to improved statistical performance of the convex function estimate when compared to the unregularized case (i.e., when $\rho=0$). {Intuitively, this is due in part to the 
behavior of the convex regression fit near the boundary of the convex hull of the covariates~\cite{balabdaoui2007consistent,mazumder2018computational}.
The $\ell_{2}$-regularization on subgradients can help regulate boundary behavior and improve performance of the estimator: see~\cite{mazumder2018computational} for theoretical support.
Here, we present some numerical evidence to support this observation.} \Cref{fig:rmse} presents the training and test root mean squared error (RMSE)\footnote{{The RMSE is computed by the following procedure: (i) obtain the primal feasible solution $(\tilde{\bm\phi},\tilde{\bm\xi})$ according to \Cref{sec:dualgap}, (ii) obtain the prediction $\hat y$ for each data point $\bm x$ via a piecewise-maximum interpolation scheme $\hat y = \max_{i} \{\tilde \phi_i+\langle \tilde\xi_i,\bm x-\bm x_i \rangle\}$; (iii) evaluate RMSE based on the predictions $\hat y$ and the observed values $y$.}} 
on some real datasets. 
To quantify the performance of the convex function estimate near the boundary of the convex hull, we compute the RMSE on the boundary points\footnote{We compute the convex hull of the training set and identify points in the test set near the boundary of this convex hull, according to the distance of each point to the convex hull.}. {\Cref{fig:rmse} shows what values of $\rho>0$ result in good statistical performance (in terms of RMSE).
In particular, values of $\rho=10^{-4}\sim10^{-5}$ generally result in good RMSE performance for the real data sets.} 
We also observe that 
$\rho=10^{-3}$-$10^{-4}$ result in good statistical performance for the synthetic data sets.
{{In practice, we recommend selecting a value of $\rho$ that minimizes RMSE on a validation dataset (or based on cross-validation).}}

\begin{figure}[ht]
    \centering
\resizebox{0.97\textwidth}{!}{\begin{tabular}{r c c}
& RD1 with $n=10,000$& RD2 with $n=10{,}000$\\
\rotatebox{90}{{{~~~~~~~~~~~~~~~~~~~RMSE}}}&\includegraphics[width =0.5 \textwidth, trim = .25cm .25cm 2cm 1.5cm,clip ]{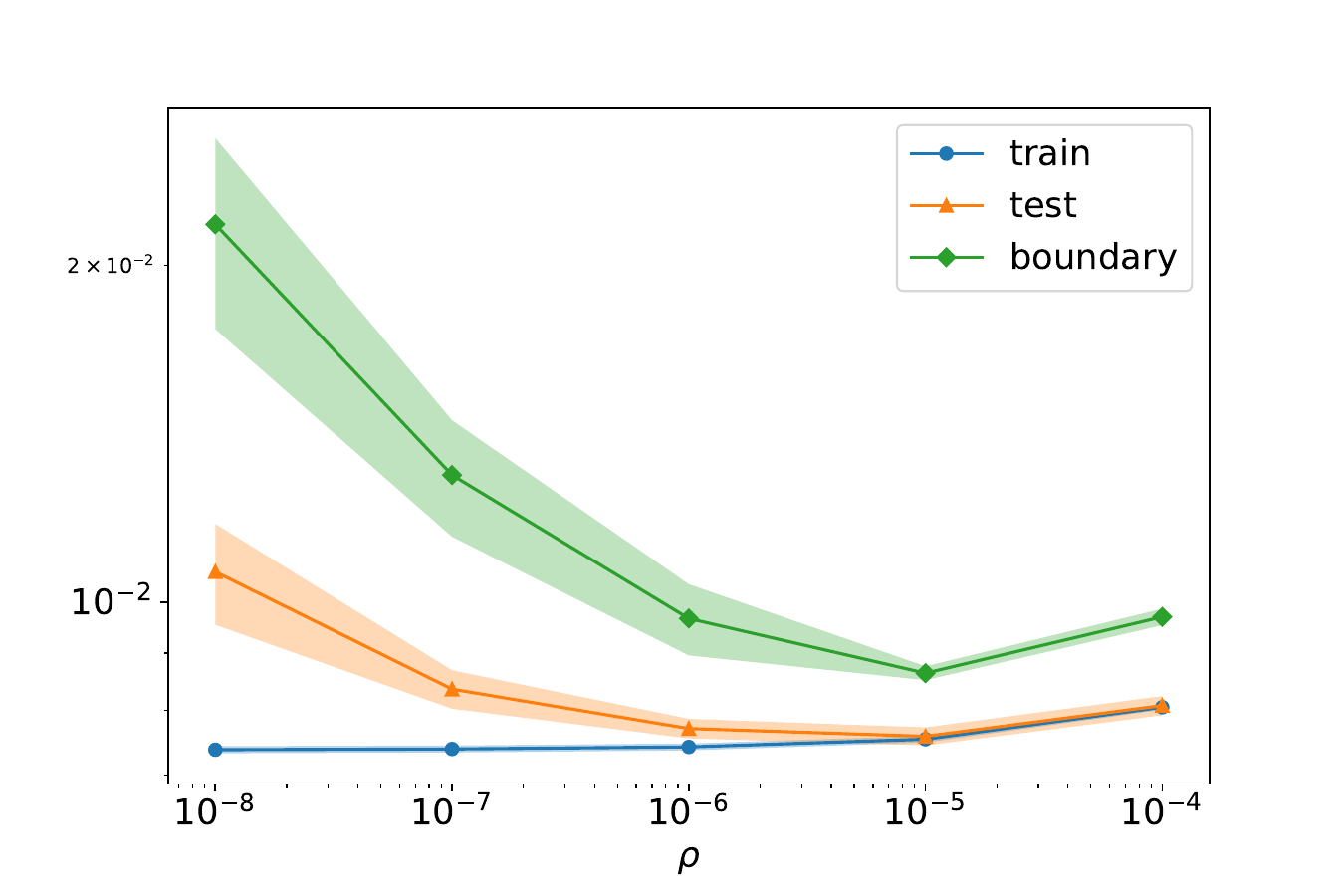}&\includegraphics[width =0.5 \textwidth, trim = .25cm .25cm 2cm 1.5cm,clip ]{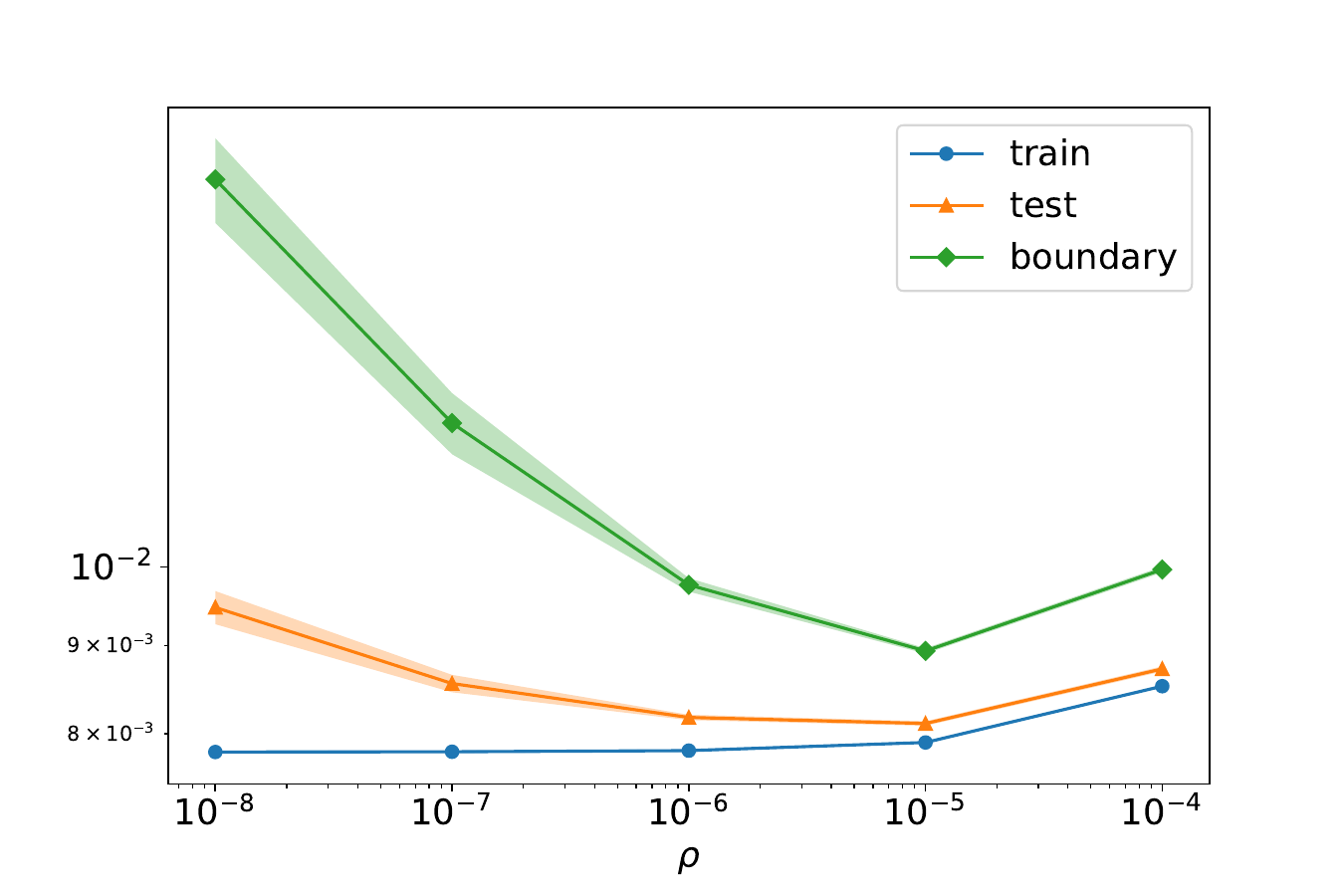}\\
& RD3 with $n=10{,}000$& RD4 with $n=10{,}000$\\
\rotatebox{90}{{{~~~~~~~~~~~~~~~~~~~RMSE}}}&\includegraphics[width =0.5 \textwidth,  trim = .25cm .25cm 2cm 1.5cm,clip ]{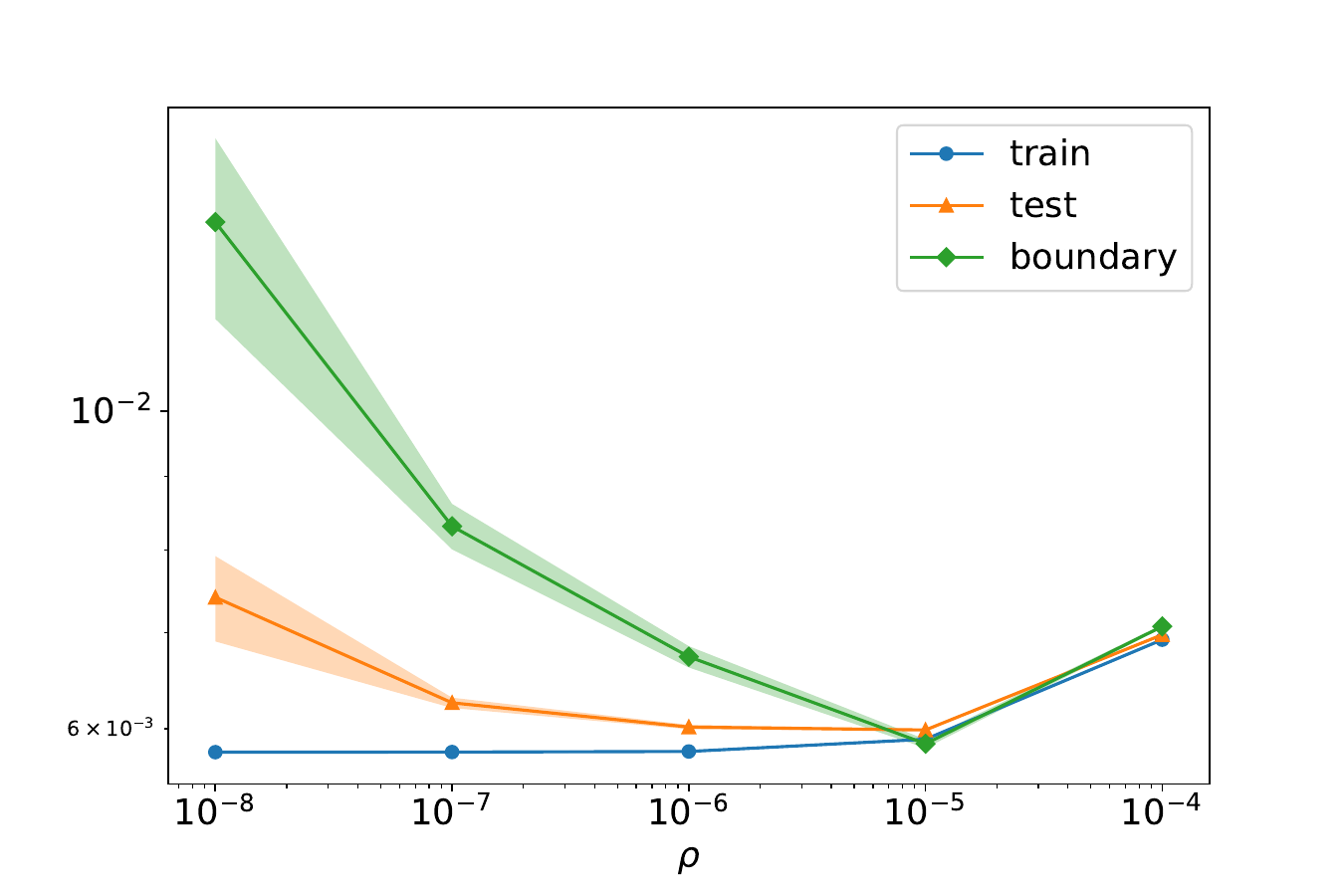}&\includegraphics[width =0.5 \textwidth, trim = .25cm .25cm 2cm 1.5cm,clip ]{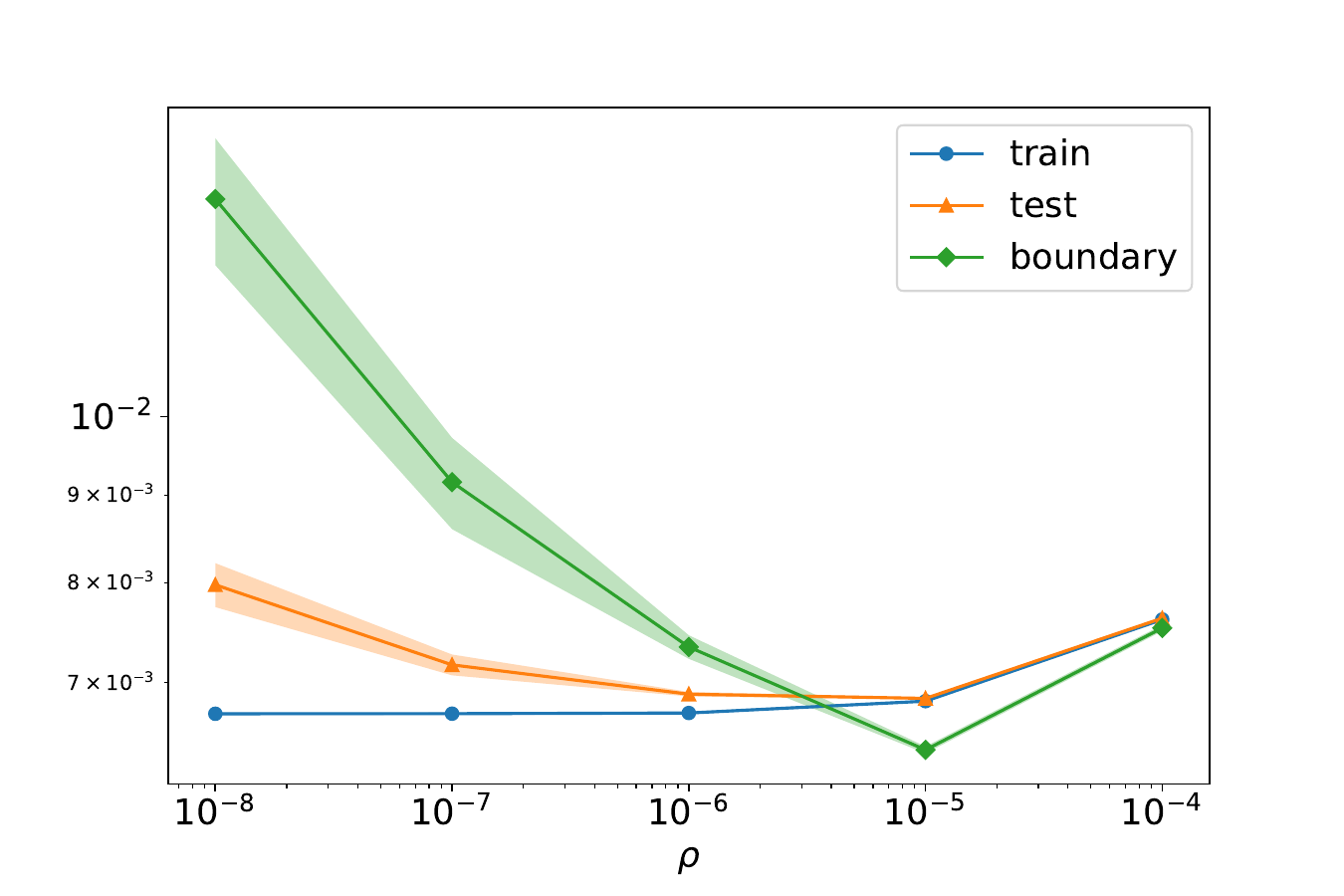}
    \end{tabular}
    }
    \caption{{\small{Plots (in log-scale) of RMSEs evaluated on training set, test set and boundary test set versus $\rho$'s for 4 real datasets. We consider ten replications (subsamples) and 
    plot the mean (markers) and standard error (error bars). The training RMSE decreases with $\rho$ and appears to stabilize when $\rho$ becomes smaller than $10^{-5}$ (approx). 
    We observe that the minimum RMSE on the test/boundary set occurs when $\rho \approx 10^{-5}$, and this value is quite close to the test RMSE at $\rho\approx 10^{-4}$. 
We study both these $\rho$-values in our runtime comparisons.}}}
    \label{fig:rmse}
\end{figure}

\section{Proofs}\label{app:proofs}

In this section, we  present the proofs of \Cref{lem:outer-conv-lem2-main} and \Cref{thm:outer-conv}. 
The proof of \Cref{thm:outer-conv} is based on \Cref{lem:outer-conv-lem2-main,lem:outer-conv-lem1}.
We first present the proof of \Cref{lem:outer-conv-lem2-main} (Sec. \ref{subsecA:proof-lem2}); then present \Cref{lem:outer-conv-lem1} (Sec. \ref{app:lemma-5-helper-thm}), followed by the proof of \Cref{thm:outer-conv}.
\subsection{Proof of \Cref{lem:outer-conv-lem2-main}}\label{subsecA:proof-lem2}
For any $\bm u\in\R^N$ and $P\leq N$, we define two norms:
\begin{equation}\label{defn-u-N-1}
\norm{\bm u}_{\mathcal{G}[P]} = \left(\sum_{i\in\mathcal{G}[[N],P,\bm u]}u_i^2\right)^{1/2}~~~\text{and}~~~\norm{\bm u}_{\mathcal{U}[P]}=\E\left(\sum_{i\in\mathcal{U}[[N],P]} u_i^2\right)^{1/2}
\end{equation}
where, $\mathcal{G}[[N],P,\bm u]$, $\mathcal{U}[P]$ are defined in~\Cref{def:indexset} and $\mathbb{E}(\cdot)$ is the expectation wrt scheme ${\mathcal U}$. \Cref{prop:norm-equiv} links these norms to the Euclidean norm:
\begin{proposition} \label{prop:norm-equiv} {For $\bm u\in\R^N$ and the norms defined in~\eqref{defn-u-N-1}}, the following statement holds:
$$\norm{\bm{ u}}^2\geq \norm{\bm{ u}}_{\mathcal{G}[P]}^2\geq \frac{P}{N}\norm{\bm{ u}}^2=\norm{\bm{ u}}_{\mathcal{U}[P]}^2.$$
\end{proposition}
\begin{proof}[Proof of Proposition~\ref{prop:norm-equiv}]
{For $\bm u\in\R^N$,} we see that $\norm{\bm{ u}}_{\mathcal{U}[P]}^2=(P/N)\norm{\bm{ u}}^2.$
Notice that
\begin{align*}
    \norm{\bm{ u}}_{\mathcal{G}[P]}^2~=~\max_{\bm{\pi}}\sum_{\omega\in[N]}\pi_\omega| u_\omega|^2~~~\text{s.t.}~~~\sum_{\omega\in[N]}\pi_\omega \leq  P, ~~~ 0\leq \pi_\omega\leq 1,\forall\omega.
\end{align*}

Since $\pi_\omega = {P}/{N}~\forall \omega$, is  feasible for the above problem, it follows that $$\norm{\bm{ u}}_{\mathcal{G}[P]}^2\geq \frac{P}{N}\norm{\bm{ u}}^2=\norm{\bm{ u}}_{\mathcal{U}[P]}^2.$$
Equality above is attained if and only if $ u_\omega = C$ $\forall \omega$ for some $C$. Furthermore, we note that
 $$\norm{\bm{ u}}_{\mathcal{G}[P]}^2=\sum_{\omega\in\mathcal{G}[[N],P]}| u_\omega|^2\leq \sum_{\omega\in[N]}| u_\omega|^2=\norm{\bm{ u}}^2,$$
and this equality is attained if and only if $ u_\omega =0$ for all $\omega\not\in \mathcal{G}[[N],P]$, i.e. the $N-P$ smallest values of $| u_\omega|$ are 0. 
\end{proof}

\begin{proof}[{Proof of \Cref{lem:outer-conv-lem2-main}}]
We divide the proof into 5 parts depending upon the 5 rules.

\medskip

\noindent \textbf{Rule 1: Greedy within each Block.} ($\delta_1(\bm\theta,\{\Omega_i\}) =\cup_{i=1}^n\mathcal{G}[\Omega_i,P,\bm\theta] $, $\alpha_{\{1\}} = \frac{n-1}{P}$, and $\beta_{\{1\}}=1$.) For this selection rule, we have
\begin{align*}
    \norm{\bm{\theta}}_{\{1\}}^2&=
    \sum_{i=1}^n\norm{\bm{\theta}_{\Omega_i}}_{\mathcal{G}[P]}^2.
\end{align*}
It is easy to see that $\|{\bm{\theta}}\|_{\{1\}}$ is a norm. It follows from \Cref{prop:norm-equiv} that 
$$\norm{\bm{\theta}}^2=\sum_{i=1}^n\norm{\bm{\theta}_{\Omega_i}}^2\geq \sum_{i=1}^n \norm{\bm{\theta}_{\Omega_i}}_{\mathcal{G}[P]}^2\geq \sum_{i=1}^n\frac{P}{n-1}\norm{\bm{\theta}_{\Omega_i}}^2=\frac{P}{n-1}\norm{\bm{\theta}}^2.$$
Therefore, we have $\alpha_{\{1\}}=\frac{n-1}{P}$, and $\beta_{\{1\}}=1.$

\medskip

\noindent \textbf{Rule 2: Random.}  ($\delta_2(\bm\theta,\{\Omega_i\})=\mathcal{U}[\Omega,K]$,  $\alpha_{\{2\}}=\frac{n(n-1)}{K}$ and $\beta_{\{2\}}=\frac{n(n-1)}{K}$).

For this selection rule, we have
$$\norm{\bm{\theta}}_{\{2\}}^2=\norm{\bm{\theta}}_{\mathcal{U}[K]}^2=\frac{K}{n(n-1)}\norm{\bm{\theta}}^2.$$
Thus, in this case the norm-equivalence constants are $\alpha_{\{2\}}=\beta_{\{2\}}=\frac{n(n-1)}{K}$.

\medskip

\noindent \textbf{Rule 3: Random within each Block.} ($\delta_3(\bm\theta,\{\Omega_i\}) = \cup_{i=1}^n\mathcal{U}[\Omega_i,P]$, $\alpha_{\{3\}}=\frac{n-1}{P}$, and $\beta_{\{3\}}=\frac{n-1}{P}$.) For this selection rule, we have
$$\norm{\bm{\theta}}_{\{3\}}^2=\sum_{i=1}^n\norm{\bm{\theta}_{\Omega_i}}_{\mathcal{U}[P]}^2 =\sum_{i=1}^n\frac{P}{n-1}\norm{\bm{\theta}_{\Omega_i}}^2=\frac{P}{n-1}\norm{\bm{\theta}}^2.$$
Hence, it follows  that $\alpha_{\{3\}}=\beta_{\{3\}}=\frac{n-1}{P}.$

\medskip 

\noindent \textbf{Rule 4: Random then Greedy.} ($\delta_4(\bm\theta,\{\Omega_i\})=\mathcal{G}[\mathcal{U}[\Omega,M],K,\bm\theta]$, $\alpha_{\{4\}}=\frac{n(n-1)}{K}$, and $\beta_{\{4\}}=\frac{n(n-1)}{M}$.) We adapt the proof of~\cite{lu2018randomized} to show that under this selection rule
$$\norm{\bm{\theta}}_{\{4\}}^2=\sum_{l=1}^{n(n-1)}\pi(l)|\theta_{(l)}|^2,$$
where $|\theta_{(l)}|$ is the $l$-th largest value in $\{|\theta_\omega|\}_{\omega\in \Omega}$ and $\pi(l)$ is given by
$$\pi(l)=\frac{M}{n(n-1)}\sum_{k=1}^{K}\frac{\binom{l-1}{k-1}\binom{n(n-1)-l}{M-k}}{\binom{n(n-1)-1}{M-1}}.$$
Here, by convention, we define $\binom{N}{\alpha}=0$ if $\alpha<0$ or $\alpha>N$.

Let $\pi(l)$ be the probability that $|\theta_{(l)}|$ is selected. Since the subsample is selected uniformly at random, it suffices to count the number of combinations that include $|\theta_{(l)}|$ and those in which $|\theta_{(l)}|$ ranks among the top $K$ values. This is equivalent to choosing $k-1(\leq K-1)$ elements from $\{|\theta_{(s)}|\}_{s\leq l-1}$, selecting the element $|\theta_{(l)}|$ and then choosing the remaining $(M-k)$ elements from the rest. Therefore, the number of such combinations is 
$$N(l)=\sum_{k=1}^{K}\binom{l-1}{k-1}\binom{n(n-1)-l}{M-k}.$$
Thus, $$\pi(l)=\frac{N(l)}{\binom{n(n-1)}{M}}=\frac{M}{n(n-1)}\sum_{k=1}^{K}\frac{\binom{l-1}{k-1}\binom{n(n-1)-l}{M-k}}{\binom{n(n-1)-1}{M-1}}.$$
Notice that when $l\leq K$  (i.e., $l-1\leq K-1$), then
$$N(l)=\sum_{k=1}^{l}\binom{l-1}{k-1}\binom{n(n-1)-l}{M-k}=\binom{n(n-1)-1}{M-1},$$
and thus $\pi(l)=\frac{M}{n(n-1)}.$ 

When $l\geq n(n-1)-(M-K)+1$ and $M-k\geq M-1-K>n(n-1)-l$, then $N(l)=0$ and thus $\pi(l)=0$. 

As each element appears in the same number of combinations of size $M$ (in the random selection step), and the greedy selection step favors the larger one, we have the following ordering: $\pi(1)\geq \ldots\geq \pi(n(n-1))$. Therefore, $\norm{\bm{\theta}}_{\{4\}}$ is a norm.

Since $\pi(l)=\E[\mathbb{I}_{\{(l)\in \mathcal{G}[\mathcal{U}[\Omega,M],K]\}}]$, it follows that
$$\sum_{l=1}^{n(n-1)}\pi(l)=\E\left[\sum_{l=1}^{n(n-1)}\mathbb{I}_{\{(l)\in \mathcal{G}[\mathcal{U}[\Omega,M],K]\}}\right]=K.$$

Then, we have
\begin{equation} \label{norm-rule4-1}
    \norm{\bm{\theta}}_{\{4\}}^2=\sum_{l=1}^{n(n-1)}\pi(l)|\theta_{(l)}|^2\geq \sum_{l=1}^{n(n-1)}\frac{K}{n(n-1)}|\theta_{(l)}|^2=\frac{K}{n(n-1)}\norm{\bm{\theta}}^2.
    \end{equation}
On the other hand, since $\pi(l)\leq \pi(1)=\frac{M}{n(n-1)}$, it follows that
\begin{equation} \label{norm-rule4-2}
\norm{\bm{\theta}}_{\{4\}}^2=\sum_{l=1}^{n(n-1)}\pi(l)|\theta_{(l)}|^2\leq \sum_{l=1}^{n(n-1)}\frac{M}{n(n-1)}|\theta_{(l)}|^2=\frac{M}{n(n-1)}\norm{\bm{\theta}}^2.
\end{equation}
Hence from~\eqref{norm-rule4-1} and~\eqref{norm-rule4-2}, we obtain $\alpha_{\{4\}}=\frac{n(n-1)}{K}$ and $\beta_{\{4\}}=\frac{n(n-1)}{M}$.

\medskip

\noindent \textbf{Rule 5: Random Blocks then Greedy within each Block.}\\ ($\delta_5(\bm\theta,\{\Omega_i\})=\cup_{i\in\mathcal{U}[[n],G]}\mathcal{G}[\Omega_i,P,\bm\theta]$, $\alpha_{\{5\}}=\frac{n(n-1)}{GP}$, and $\beta_{\{5\}}=\frac{n}{G}$.)

By the selection rule, we have
$$\norm{\bm{\theta}}_{\{5\}}^2=\frac{G}{n}\sum_{i=1}^n\norm{\bm{\theta}_{\Omega_i}}_{\mathcal{G}[P]}^2=\frac{G}{n}\norm{\bm{\theta}}_{\{1\}}^2$$
with $\alpha_{\{5\}}=\frac{n}{G}\alpha_{\{1\}}=\frac{n(n-1)}{GP}$ and $\beta_{\{5\}}=\frac{n}{G}\beta_{\{1\}}=\frac{n}{G}.$

\end{proof}

\subsection{Auxiliary lemmas for the proof of \Cref{thm:outer-conv}}\label{app:lemma-5-helper-thm}
Here we present \Cref{lem:conv-proof-geq,lem:outer-conv-lem1} that will be used for the proof of \Cref{thm:outer-conv}. 

\begin{lemma}[Hoffman's Lemma \cite{hoffman1952approximate}] 
Let $\Lambda =\{\bm{\lambda}:\bm{D\lambda}=\bm{s},\bm{\lambda}\leq \M{0}\}$. There is a constant $\mu_{\bm D}>0$ that depends only on $\bm D$ such that for any $\bm{\lambda}\leq \M{0}$, there exists an $\bm{\lambda}_0\in \Lambda$ with
$$\norm{\bm{D\lambda}-\bm{s}}^2\geq \mu_{\bm D}\norm{\bm{\lambda}-\bm{\lambda}_0}^2.$$
The constant $\mu_{\bm D}$ is called the Hoffman's constant associated with $\bm D$.
\end{lemma}

Let $\bm{D}^\star=\begin{bmatrix}\bm A^\top\\\bm{B}^\top/\sqrt{\rho}\end{bmatrix}$ be the coefficient matrix corresponding to set of optimal solutions $\B\Lambda^\star$ with $W=\Omega$ in \eqref{eqn:opt-dual-poly-W}. We get the following sufficient descent lemma according to \cite{karimi2016linear}:
\begin{lemma}[\cite{karimi2016linear} Case 3, Page 18]\label{lem:conv-proof-geq}
For any $\bm{\lambda}\leq \M{0}$ and optimal dual solution $\bm{\lambda}^\star$ of \eqref{eqn:main-dual}, we have
\begin{equation}\label{eqn:conv-proof-geq}
    L(\bm{\lambda}^\star)\geq L(\bm{\lambda})-\frac{1}{2\mu_{\bm{D}^\star}}\mathcal{D}(\bm{\lambda},\nabla L(\bm{\lambda}),\mu_{\bm{D}^\star}),
\end{equation}
where $\mu_{\bm{D}^\star}>0$ is the Hoffman's constant associated with $\bm{D}^\star$.
\end{lemma}

\begin{definition}\label{defn-defn-mathcalD}
Define a function $\mathcal{D}(\cdot,\cdot,\cdot):\mathbb{R}_-^N\times\mathbb{R}^N\times \mathbb{R}_{++}\to \mathbb{R}_+$ as
\begin{equation*}
    \mathcal{D}(\bm{\lambda},\bm{\theta},\mu)=-2\mu \min_{\bm{\lambda}'\leq 0}\left\{\langle\bm{\theta},\bm{\lambda}'-\bm{\lambda}\rangle+\frac{\mu}{2}\norm{\bm{\lambda}'-\bm{\lambda}}^2\right\}.
\end{equation*}
\end{definition}

\Cref{lem:outer-conv-lem1} makes use of the following useful proposition.
\begin{proposition}\label{prop:d-func}
Define $d(\cdot,\cdot,\cdot):\mathbb{R}_-\times \mathbb{R}_+\times \mathbb{R}_{++}\to\mathbb{R}$ as
$$d(\lambda,\theta, \mu)=-2\mu\min_{\lambda'\leq 0}\left\{\theta(\lambda'-\lambda)+\frac{\mu}{2}(\lambda'-\lambda)^2\right\}.$$
We have the following:
\begin{itemize}
    \item[(a)] if $\lambda=0$, then for any $\mu>0$, $d(\lambda,\theta,\mu)=\max\{\theta,0\}^2;$
    \item[(b)] if $\lambda\theta = 0 $ with $\lambda \leq 0$ and $\theta \leq 0$, then for any $\mu>0 $, $d(\lambda,\theta,\mu)=0$;
    \item[(c)] if $\lambda\leq 0$, then for any $\mu>0$, $d(\lambda,\theta,\mu)\geq 0$.
\end{itemize}
\end{proposition}
\begin{proof}
\noindent Part (a):~~If $\lambda =0$, then $$d(\lambda,\theta,\mu)=-2\mu\min_{\lambda'\leq 0}\left\{\theta \lambda'+\frac{\mu}{2}(\lambda')^2\right\}=\left\{\begin{array}{ll}0&\text{ if }\theta \leq 0\\
\theta^2&\text{ if }\theta >0
\end{array}\right.,$$
i.e. $d(\lambda,\theta,\mu)=\max\{\theta,0\}^2$.

\noindent Part (b):~~If $\lambda\theta =0$ with $\lambda=0$ and $\theta\leq 0$, then $d(\lambda,\theta,\mu)=0$ follows from Part (a).

If $\lambda \theta= 0$ with $\lambda <0$, then $\theta = 0$, so $d(\lambda,\theta,\mu)=-2\mu\min_{\lambda'\leq 0}\{\frac{\mu}{2}(\lambda'-\lambda)^2\}=0.$

\noindent Part (c): If $\lambda \leq 0$, then $\lambda'=\lambda\leq 0$ is always feasible, so $d(\lambda,\theta,\mu)\geq 0.$
\end{proof}

\begin{lemma}\label{lem:outer-conv-lem1}
Suppose there exists $\alpha\geq1$ such that for any $m$ and $W^0$ the following
\begin{equation}\label{assume-lemma=proof11}
    \alpha \E_{\eta_{m+1}}\left[\sum_{\omega\in W^{m+1}\backslash W^m}\max\{\nabla_{\lambda_\omega}L(\bm{\lambda}^m),0\}^2\right]\geq \sum_{\omega\in \Omega\backslash W^m}\max\{\nabla_{\lambda_\omega}L(\bm{\lambda}^m),0\}^2
    \end{equation}
holds almost surely, where $\E_{\eta_{m+1}}$ denotes expectation over all the random sources in $(m+1)$-th iteration (conditional on events up to iteration $m$). 
Then, the following holds:
$$\E_{\eta_{m+1}}[L(\bm{\lambda}^{m+1})-L^\star]\leq \left(1-\frac{\mu_{\bm D^\star}}{\alpha\sigma}\right)[L(\bm{\lambda}^m)-L^\star].$$
\end{lemma}
\begin{proof}[Proof of Lemma~\ref{lem:outer-conv-lem1}]
We make use of the following notation in the proof: 
$$\text{$\lambda_\omega^m$ is the $\omega$-th component of $\bm{\lambda}^m$},~~~~~\nabla^m_\omega = \nabla_{\lambda_\omega}L(\bm{\lambda}^m)~~~\text{and}~~~~\tilde{\nabla}^m_\omega= \max\{\nabla^m_\omega,0\}.$$
The proof has three steps, and we consider both exact and inexact cases in each step.

\noindent {\bf Step 1:} In the first part, we will show that
\begin{align}\label{main-results-step11}
  \nonumber & L(\bm{\lambda}^m)-  L(\bm{\lambda}^\star)~~~~~~~~~~~~~~~~~~~~~~~~~~~~~~~~~~~~~~~~~~~~~~~~~~~~~~~~~~~~~ \\
   \leq & \left\{\begin{array}{ll}\displaystyle\frac{1}{2\mu_{\bm D^\star}}\sum_{\omega\in \Omega\backslash W^m}(\tilde{\nabla}_\omega^m)^2&\text{ (exact case)}\\
\displaystyle\frac{1}{2\mu_{\bm D^\star}}\left[\sum_{\omega\in W^m}d(\lambda^m_\omega,\nabla^m_\omega,\mu_{\bm D^\star})+\sum_{\omega\in \Omega\backslash W^m}(\tilde{\nabla}_\omega^m)^2\right]&\text{ (inexact case)}.
\end{array}\right. 
\end{align}

It follows from \Cref{lem:conv-proof-geq} that
\begin{equation}
    L(\bm{\lambda}^\star)\geq L(\bm{\lambda}^m)-\frac{1}{2\mu_{\bm D^\star}}\mathcal{D}(\bm{\lambda}^m,\nabla L(\bm{\lambda}^m),\mu_{\bm D^\star}).
\end{equation}
By the definitions of ${\mathcal D}(\cdot)$ (\Cref{defn-defn-mathcalD}) and $d(\cdot)$ (\Cref{prop:d-func}),
we have 
\begin{align}
    \mathcal{D}(\bm{\lambda}^m,\nabla L(\bm{\lambda}^m),\mu_{\bm D^\star})&=\sum_{\omega\in \Omega}d(\lambda^m_\omega,\nabla^m_\omega,\mu_{\bm D^\star})\nonumber\\
    &=\sum_{\omega\in W^m}d(\lambda^m_\omega,\nabla^m_\omega,\mu_{\bm D^\star})+\sum_{\omega\in \Omega\backslash W^m}d(\lambda^m_\omega,\nabla^m_\omega,\mu_{\bm D^\star})
\end{align}
\subparagraph{Inexact Case:} By the definition of $\bm{\lambda}^{W^m}$, we know that $\lambda_\omega^m=0$ for $\omega\not\in W^m$. Then it follows from \Cref{prop:d-func} that 
\begin{align}\label{step1-inexact-case-11}
    \mathcal{D}(\bm{\lambda}^m,\nabla L(\bm{\lambda}^m),\mu_{\bm D^\star})
    =\sum_{\omega\in W^m}d(\lambda^m_\omega,\nabla^m_\omega,\mu_{\bm D^\star})+\sum_{\omega\in \Omega\backslash W^m}(\tilde{\nabla}_\omega^m)^2
\end{align}
\subparagraph{Exact Case:} By the primal feasibility, dual feasibility and complementary slackness of $\bm{\lambda}_{W^m}^\star$, we have $\lambda_\omega^m\nabla_\omega^m=0$ with $\lambda_\omega^m,\nabla_\omega^m\leq 0$ for $\omega\in W^m$. Combining this with $\lambda_\omega^m=0$ for $\omega\not\in W^m$, by \Cref{prop:d-func}, we have
\begin{equation}\label{step1-exact-case-11}
\mathcal{D}(\bm{\lambda}^m,\nabla L(\bm{\lambda}^m),\mu_{\bm D^\star})=\sum_{\omega\in \Omega\backslash W^m}(\tilde{\nabla}_\omega^m)^2.
\end{equation}
 The conclusions from~\eqref{step1-inexact-case-11} and~\eqref{step1-exact-case-11} lead to~\eqref{main-results-step11}.

\noindent {\bf Step 2:} In the second step, we will show that 
\begin{align}\label{main-results-step22}
\nonumber    &L(\bm{\lambda}^{m+1})-L(\bm{\lambda}^m)~~~~~~~~~~~~~~~~~~~~~~~~~~~~~~~ \\
    \leq& \left\{\begin{array}{ll}\displaystyle-\frac{1}{2\sigma}\sum_{\omega\in W^{m+1}\backslash W^m}(\tilde{\nabla}_\omega^m)^2&\text{ (exact case)}\\
\displaystyle-\frac{1}{2\sigma}\left[\sum_{\omega\in W^m}d(\lambda^m_\omega,\nabla^m_\omega,\sigma)+\sum_{\omega\in W^{m+1}\backslash W^m}(\tilde{\nabla}_\omega^m)^2\right]&\text{ (inexact case)}
\end{array}\right.
\end{align}

\subparagraph{Exact Case:} Let $\Lambda^{m+1}=\{\bm{\lambda}\in\mathbb{R}^{n(n-1)}:\lambda_\omega=0,\forall \omega\in\Omega\backslash W^{m+1}\}.$ 
Recall that $\bm{\lambda}^{m+1}$ minimizes $L(\B\lambda)$ over $\Lambda^{m+1}$.
Therefore, we have the following:
\begin{align}
    L(\bm{\lambda}^{m+1})&=\min_{\bm{\lambda}\in \Lambda^{m+1}}L(\bm{\lambda})\nonumber\\
    &\leq \min_{\bm{\lambda}\in \Lambda^{m+1}}  \left\{L(\bm{\lambda}^m)+\langle \nabla L(\bm{\lambda}^m),\bm{\lambda}-\bm{\lambda}^m\rangle +\frac{\sigma}{2}\norm{\bm{\lambda}-\bm{\lambda}^m}^2 \right\} \nonumber\\
    &=L(\bm{\lambda}^m)-\frac{1}{2\sigma}\sum_{\omega\in W^m}d(\lambda^m_\omega,\nabla^m_\omega,\sigma)-\frac{1}{2\sigma}\sum_{\omega\in W^{m+1}\backslash W^m}d(\lambda^m_\omega,\nabla^m_\omega,\sigma)\nonumber\\
    &=L(\bm{\lambda}^m)-\frac{1}{2\sigma}\sum_{\omega\in W^{m+1}\backslash W^m}(\tilde{\nabla}_\omega^m)^2 \label{exact-case-step2-11}
\end{align}
where the first inequality uses $\sigma$-smoothness of $L$; the last line follows from an argument similar to~\eqref{step1-exact-case-11} where we use \Cref{prop:d-func} and complementary slackness.

\subparagraph{Inexact Case:} Here we take one (or more) projected gradient step(s) to partially minimize the reduced dual. Let $\bm{\lambda}^{m+\frac{1}{2}}$ be obtained after one projected gradient descent step from $\bm{\lambda}^m$. Then we have
\begin{align}
    L(\bm{\lambda}^{m+1})&\leq L(\bm{\lambda}^{m+\frac{1}{2}})\nonumber\\
    &\leq L(\bm{\lambda}^m)+\langle \nabla L(\bm{\lambda}^m),\bm{\lambda}^{m+\frac{1}{2}}-\bm{\lambda}^m\rangle +\frac{\sigma}{2}\norm{\bm{\lambda}^{m+\frac{1}{2}}-\bm{\lambda}^m}^2\nonumber\\
    &=L(\bm{\lambda}^m)-\frac{1}{2\sigma}\sum_{\omega\in W^{m+1}}d(\lambda^m_\omega,\nabla^m_\omega,\sigma)\nonumber\\
    &= L(\bm{\lambda}^m)-\frac{1}{2\sigma}\sum_{\omega\in W^m}d(\lambda^m_\omega,\nabla^m_\omega,\sigma)-\frac{1}{2\sigma}\sum_{\omega\in W^{m+1}\backslash W^m}(\tilde{\nabla}_\omega^m)^2,\label{inexact-case-step2-11}
\end{align}
where the first inequality follows from the descent property of projected gradient descent; the second inequality uses $\sigma$-smoothness of $L$; and the last line follows from \Cref{prop:d-func}.

Finally, the result in~\eqref{main-results-step22} follows from~\eqref{exact-case-step2-11} and~\eqref{inexact-case-step2-11}

\noindent {\bf Step 3:} For the third step, we will show that the following holds
\begin{equation}\label{main-result-step3}
   \sigma\alpha \E_{\eta_{m+1}}[L(\bm{\lambda}^{m+1})-L(\bm{\lambda}^m)]\leq -\mu_{\bm D^\star}[ L(\bm{\lambda}^m) -L(\bm{\lambda}^\star)].
\end{equation}
\subparagraph{Exact Case:} For this case, we have the following chain of inequalities:
\begin{align}
     -2\sigma\alpha \E_{\eta_{m+1}}[L(\bm{\lambda}^{m+1})-L(\bm{\lambda}^m)]&
    \geq \alpha\E_{\eta_{m+1}}\sum_{\omega\in W^{m+1}\backslash W^m}(\tilde{\nabla}_\omega^m)^2\nonumber\\
    &\geq \sum_{\omega\in \Omega\backslash W^m}(\tilde{\nabla}_\omega^m)^2\nonumber\\
    &\geq 2\mu_{\bm D^\star}[ L(\bm{\lambda}^m) -L(\bm{\lambda}^\star)], \label{step3-exact-case-11}
\end{align}
where the first inequality uses \eqref{main-results-step22}; the second inequality is the assumption \eqref{assume-lemma=proof11}; and the last line uses \eqref{main-results-step11}.

\subparagraph{Inexact Case:} Since $\lambda_\omega^m\leq 0$, we know $d(\lambda_\omega^m,\nabla_\omega^m,\sigma)\geq 0$ by \Cref{prop:d-func}. Using the fact that $\alpha\geq 1$, we have the following:
\begin{align}
     -2\sigma\alpha \E_{\eta_{m+1}}[L(\bm{\lambda}^{m+1})-L(\bm{\lambda}^m)]&
    \geq \alpha\sum_{\omega\in W^m}d(\lambda^m_\omega,\nabla^m_\omega,\sigma)+  \alpha\E_{\eta_{m+1}}\sum_{\omega\in W^{m+1}\backslash W^m}(\tilde{\nabla}_\omega^m)^2\nonumber\\
    &\geq \sum_{\omega\in W^m}d(\lambda^m_\omega,\nabla^m_\omega,\sigma) + \sum_{\omega\in \Omega\backslash W^m}(\tilde{\nabla}_\omega^m)^2\nonumber\\
    &\geq 2\mu_{\bm D^\star}[ L(\bm{\lambda}^m) -L(\bm{\lambda}^\star)],\label{step3-inexact-case-11}
\end{align}
where the first inequality uses \eqref{main-results-step22}; the second inequality uses assumption \eqref{assume-lemma=proof11}, $\alpha\geq 1$ and $d(\lambda_\omega^m,\nabla_\omega^m,\sigma)\geq 0$; and the last line uses \eqref{main-results-step11}.

The statement in~\eqref{main-result-step3} follows from~\eqref{step3-exact-case-11} and~\eqref{step3-inexact-case-11}.

Finally, we complete the proof by using~\eqref{main-result-step3} and observing that:
\begin{align*}
    \E_{\eta_{m+1}}[L(\bm{\lambda}^{m+1})-L^\star]=&  \E_{\eta_{m+1}}[L(\bm{\lambda}^{m+1})-L(\bm{\lambda}^m)]+ L(\bm{\lambda}^m)-L^\star \\
    \leq& \left(1-\frac{\mu_{\bm D^\star}}{\alpha\sigma}\right)[L(\bm{\lambda}^m)-L^\star].
\end{align*}
\end{proof}

\subsection{Proof of \Cref{thm:outer-conv}}\label{subsec:proof-outer-thm}
The proof of \Cref{thm:outer-conv} uses Lemmas~\ref{lem:outer-conv-lem2-main} and \ref{lem:outer-conv-lem1}.
\begin{proof}[\textbf{Proof of \Cref{thm:outer-conv}}]
Recall that $\lambda_\omega^m$ is the $\omega$-th component of $\bm{\lambda}^m$,  $\nabla^m_\omega = \nabla_{\lambda_\omega}L(\bm{\lambda}^m)$, and $\tilde{\nabla}^m_\omega= \max\{\nabla^m_\omega,0\}$. Now let $\{\Omega_i\}$ be one of the partitions $\{\Omega_{i\cdot}\}$ or $\{\Omega_{\cdot i}\}$, and $W_i^m$ be the corresponding partition for $W^m$. Let $\Delta=W^{m+1}\backslash W^m$.
Using this notation, the condition of \Cref{lem:outer-conv-lem1} reduces to 
\begin{equation}\label{lemma5-cond-simplified1}
\alpha \E_{\eta_{m+1}}[\sum_{\omega\in\Delta}(\tilde{\nabla}_\omega^m)^2]\geq \sum_{\omega \in \Omega\backslash W^m}(\tilde{\nabla}_\omega^m)^2.
    \end{equation}
    
We organize the proof into four steps: (1) we construct a random vector $\bm{g}\in\mathbb{R}^{n(n-1)}$ from $\bm{\nabla}^m = \nabla L(\bm{\lambda}^m)$; (2) we then show $\sum_{\omega\in\Omega\backslash W^m}(\tilde{\nabla}_\omega^m)^2$ equals $\norm{\bm{g}}^2$; 
(3) we relate $\E_{\eta_{m+1}}[\sum_{\omega\in\Delta}(\tilde{\nabla}_\omega^m)^2]$ to $\norm{\bm{g}}_{\{\ell\}}^2$; and (4) finally, we apply \Cref{lem:outer-conv-lem2-main,lem:outer-conv-lem1} to complete the proof.

\medskip

\noindent \textbf{Step 1:} Define each entry $g_\omega$ of $\bm{g}$ as follows:
\begin{equation}\label{defn-g-deltaw}
    g_\omega=\left\{\begin{array}{ll} \tilde{\nabla}_\omega^m&\text{ if }\omega\in \Omega\backslash W^m\\
    0&\text{ if }\omega\in  W^m
    \end{array}\right..
\end{equation}
Notice $\bm{g}$ is a random vector 
depending upon random sources from iterations $1, \ldots, m$.

\medskip

\noindent\textbf{Step 2:} By the definition of $\bm{g}$, we have
$\sum_{\omega\in\Omega\backslash W^m}(\tilde{\nabla}_\omega^m)^2 = \norm{\bm{g}}_2^2.$

\medskip

\noindent \textbf{Step 3:} Recall that $\Delta'$ denotes the set of selected pairs as per a rule, and we consider a 
subset 
$\Delta=\{\omega\in\Delta':\tilde{\nabla}_\omega^m=\nabla_\omega^m > 0\}$---that is, $\tilde{\nabla}_\omega^m=0$ for any $\omega\in\Delta'\backslash \Delta$. Thus,
$$\E_{\eta_{m+1}}\left[\sum_{\omega\in \Delta}(\tilde{\nabla}_\omega^m)^2\right]=\E_{\eta_{m+1}}\left[\sum_{\omega\in \Delta'}(\tilde{\nabla}_\omega^m)^2\right].$$

Note that $\Delta'=\delta_{\ell}(\tilde{\bm{\nabla}}^m_{\Omega\backslash W^m},\{\Omega_i\setminus W_i^m\})$, and let $\Delta''= \delta_{\ell}(\bm g,\{\Omega_i\})$.
Notice that $\M{g}$ has more zero coordinates compared to $\tilde{\bm{\nabla}}^m$ (see~\eqref{defn-g-deltaw}).
Thus, we have
$$\E_{\eta_{m+1}}\left[\sum_{\omega\in \Delta}(\tilde{\nabla}_\omega^m)^2\right]=\E_{\eta_{m+1}}\left[\sum_{\omega\in \Delta'}(\tilde{\nabla}_\omega^m)^2\right]\geq \E_{\eta_{m+1}}\left[\sum_{\omega\in \Delta''}(g_\omega)^2\right]=\norm{\bm{g}}_{\{\ell\}}^2,$$
where the last equality follows from \Cref{def:rule-norms}.

\medskip

\noindent \textbf{Step 4:} From Step 2, Step 3 and \Cref{lem:outer-conv-lem2-main}, we arrive at~\eqref{lemma5-cond-simplified1}. 
Therefore, it follows from \Cref{lem:outer-conv-lem1} that
$$\E_{\eta_{m+1}}[L(\bm{\lambda}^m)-L^\star]\leq\left(1-\frac{\mu_{\bm D^\star}}{\alpha_{\{\ell\}}\sigma}\right)[L(\bm{\lambda}^m)-L^\star]$$
and (using tower property of expectation) we arrive at the conclusion of the theorem:
$$\E[L(\bm{\lambda}^m)-L^\star]\leq\left(1-\frac{\mu_{\bm D^\star}}{\alpha_{\{\ell\}}\sigma}\right)^m[L(\bm{\lambda}^0)-L^\star].$$

\end{proof}
\begin{remark}\label{remark:FISTA-LBFGS}
Both accelerated gradient methods (APG) and L-BFGS can be used to solve the reduced problems to optimality---so, the proof for the exact subproblem optimization (in \Cref{thm:outer-conv}) for these cases will be the same as that for PGD, The proof for the inexact subproblem optimization for PGD uses only the sufficient decrease condition of the first PGD update and the fact that PGD is a descent algorithm. Since APG with adaptive restart (function scheme) \cite{o2015adaptive} and L-BFGS \cite{liu1989limited} are descent algorithms, the theory presented above, goes through as long as the progress made by APG or L-BFGS is at least as large as the progress made by the first step of APG. To this end, for L-BFGS, when setting the initial inverse Hessian approximation as the identity matrix, i.e. $\bm H_0=\bm I$, the first step of L-BFGS is exactly PGD update, so the theory works for L-BFGS. For APG with restarts, if we perform a PGD update preceding the APG updates, then the theory will go through as well.
\end{remark}

\section{Conclusion}

We present large-scale algorithms for subgradient regularized multivariate convex regression, a problem of key importance in nonparametric regression with shape constraints. We present an active set type method on the smooth dual problem~\eqref{eqn:main-dual}: we allow for inexact optimization of the reduced sub-problems and use randomized rules for augmenting the current active set. We establish novel computational guarantees for our proposed algorithms. For large-scale instances, our approach appears to be more suited to obtain low-moderate accuracy solutions.
Exploiting problem-structure, our open source toolkit can deliver approximate solutions for instances with $n \approx 10^5$ and $d=10$ (a QP with 10 billion decision variables) within a few minutes on a modest computer. Our approach appears to work well as long as the tuning parameter $\rho$ is not too small, while still corresponding to good statistical models. 
For solving~\eqref{eqn:main-primal}
with $\rho=0$ (or numerically very close to zero), our approach would not apply and we recommend using the cutting plane procedure of~\cite{bertsimas2021sparse} though this does not have computational guarantees.

\subsection*{Acknowledgements}
We would like to thank Haihao Lu for discussions in the early stages of the work. {{We also thank the referees and the Associate Editor for their comments that helped improve the paper.}} 
The authors were partially supported by ONR-N000142112841, ONR-N000141812298, 
ONR-N000142212665, NSF-IIS-1718258 and MIT-IBM Watson AI Lab.

\bibliographystyle{siamplain}

\bibliography{references}

\begin{appendix}
\section*{Appendix and Supplementary Material}
\section{Additional Technical Details}
\subsection{Examples of unavoidable factor $\alpha_{\{\ell\}}$}\label{subsubsec:unavoidable}
For simplicity, we consider an unconstrained problem with objective $f(x)=x^\top Hx$, where $H=\mathrm{diag}(\sigma,1,1,\ldots, 1 )\in\R^{p\times p}$ with $\sigma>1$. Then, $f$ is $1$-strongly convex and $\sigma$-smooth, and thus the step size will be $1/\sigma$.

\noindent {\it{Gradient Descent:}} Starting from any $x^k=(x_i^k:i\in[p])$, the gradient step yields
$$x^{k+1}=(0,\gamma x^k_2,\ldots, \gamma x^k_m),$$
where $\gamma =1-\frac1\sigma$. Therefore, for $k\geq 1$, $x^{k+1}_1=x^k_1=0$, and $f(x^{k+1})= \gamma^2f(x^k)$. Since the minimum objective value $f^*=0$, we have 
\begin{equation}\label{GD-rate}
    f(x^{k+1})-f^*=(1-(1-\gamma^2))(f(x^k)-f^*)
    \end{equation}

\noindent {\it{Our algorithm (inexact optimization)}:} We now consider the inexact active set algorithm (denoted by ASGD). Suppose, we start with initial active set $\mathcal{W}^0 = \{1\}$ and $x^0=(1,1,\ldots,1)$. We take a gradient step over $\mathcal{W}^0$ to obtain $x^1=(0,1,\ldots,1)$. Starting with $\mathcal{W}^{k-1}$ and $x^k$, we randomly select $i_k$ from $[p]\backslash\mathcal{W}^{k-1}$, and augment $\mathcal{W}^{k-1}$ with $i_k$, i.e. $\mathcal{W}^k=\mathcal{W}^{k-1}\cup\{i_k\}$. Then, we take a gradient step from $x^k$ over $\mathcal{W}^k$ to obtain $x^{k+1}$, it is easy to see that
$$x^{k+1}_i=\gamma x^k_i,\forall i\in\mathcal{W}^k,\quad \text{and}\quad x^{k+1}_i=x^k_i,\forall i\notin \mathcal{W}^k.$$
Moreover, by induction, we have that $x^{k+1}_1= 0$, $x^{k+1}$ contains $(p-k-1)$ coordinates of 1, and $k$ coordinates that are $\gamma, \gamma^2, \ldots,\gamma^k$. Therefore,
$$f(x^k)= (p-k-1)+\sum_{j=1}^k\gamma^{2k}=(p-k-1)+\frac{\gamma^2-\gamma^{2k+2}}{1-\gamma^2}.$$ Since $f^*=0$, we have for $k\geq 1$
$$\frac{f(x^k)-f(x^{k+1})}{f(x^k)}=\frac{1-\gamma^{2k+2}}{p-k-1+\frac{\gamma^2-\gamma^{2k+2}}{1-\gamma^2}}\geq \frac{1-\gamma^2}{p-1}.$$
The above inequality is tight, and it indicates
\begin{equation}\label{ASGD-rate}
    f(x^{k+1})-f^*\leq \left(1-\frac{1-\gamma^2}{p-1}\right)(f(x^k)-f^*)
\end{equation}
Compared to the rate of gradient descent in \eqref{GD-rate}, the above rate \eqref{ASGD-rate} has an additional $O(\frac{1}{p})$ factor. This is similar to having the factor $\alpha_{\{\ell\}}$ for the augmentation rules in Theorem~\ref{thm:outer-conv}. 

\section{Additional Experiment Details}\label{app:more-expts}

\subsection{Real dataset details}\label{app:real-data-details} 
{\bf RD1, RD2:} These are taken from~\url{https://archive.ics.uci.edu/ml/datasets/Gas+Turbine+CO+and+NOx+Emission+Data+Set}, and were studied in~\cite{kaya2019predicting}. The dataset has $36{,}733$ samples. 
RD1 has CO as response with features: AP, AFDP, GTEP, CDP. RD2 has NOx as response with features: AT, AP, AH, AFDP. We apply the log-transform on the responses. Each training set has $n=10,000$ and $d=4$,
with the remaining set aside for testing.\\
{\bf RD3, RD4:} These two datasets are taken from~\url{https://archive.ics.uci.edu/ml/datasets/Beijing+Multi-Site+Air-Quality+Data}; and have been studied in~\cite{zhang2017cautionary}. {This dataset has approximately $420,768$ samples; we select SO2, NO2, CO, O3 as features.} We apply the log-transform on all features and both responses PM2.5 and PM10. Each training dataset has $n=10,000$ and $d=4$ (remaining samples used for testing).\\
{\bf RD5:} This dataset is taken from~\url{https://archive.ics.uci.edu/ml/datasets/Beijing+PM2.5+Data}; and has been studied in~\cite{liang2015assessing}. 
We use PM2.5 as response with features: DEWP, TEMP, PRES, Iws. We use a log-transform of all features and response; and consider a training set with $n=10,000$ and $d=4$. \\
{\bf RD6:} This dataset is taken from~\url{https://archive.ics.uci.edu/ml/datasets/combined+cycle+power+plant}; and has been studied in~\cite{kaya2012local,tufekci2014prediction}. 
We consider a training set with $n=5,000$ and $d=4$. \\
{\bf RD7:} This dataset taken from~\cite{mekaroonreung2012estimating}, is available at \url{http://ampd.epa.gov/ampd/}, and has been recently used in~\cite{mazumder2018computational,bertsimas2021sparse}. 
We apply the log-transformation on all features and the response, and consider a training set with $n=30,000$ and $d=4$.\\
\noindent {\bf RD8:} This is taken from~\cite{ramsey2012statistical} and is the \code{ex1029} dataset available from \code{R} package \code{Sleuth2}---see also~\cite{hannah2013multivariate,balazs2016convex}. 
We first winsorize the data by excluding the samples that have response with score $\geq 2$; and consider a training set with $n=15,000$ and $d=4$.
Following~\cite{balazs2016convex}, we apply the transformation $x\mapsto 1.2^x$ to the education variable.

\subsection{Algorithm Parameters}\label{subsec:algo-params}
The tolerance for violations of constraints is set as $10^{-4}$ ($10^{-8}$ for the second stage). For inexact optimization, the tolerance for the relative objective change is $10^{-6}$ and the maximum number of PGD iterations is taken to be 5. For exact optimization, the violation tolerance is $10^{-7}$, the maximum number of PGD iterations for the sub-problems is $3,000$ and the minimum number of PGD iterations for the sub-problems is $5$.

In Rule 2, we take $K=n$; in Rule 4, we take $M=4n$ and $K=n$. In Rule 1, we take $P=1$; in Rule 5, we take $G=n/4$, and $P=4$.

In the second stage, we apply the occasional rule 1/5 when: (i) the number of constraints added by random rules 2/4 is less than $0.005n$ for consecutive 5 iterations; or (ii) the number of PGD iterations in the subproblems 
is the minimum number $5$ for consecutive 5 iterations.

Scripts used to run the experiments containing all algorithm parameters can be found in our github repository.

\end{appendix}
\end{document}